\documentclass{elsarticle}

\usepackage{hyperref}
\hypersetup{
	       colorlinks,
	       linkcolor={red!50!black},
	       citecolor={blue!50!black},
	       urlcolor={blue!80!black}
}

\journal{Linear Algebra and its Applications}

\usepackage[colorinlistoftodos,prependcaption,textsize=footnotesize]{todonotes}
\usepackage[normalem]{ulem}
\usepackage{soul}
\usepackage{amsmath}
\usepackage{amsthm}
\usepackage{amsfonts}
\usepackage{amssymb}         
\usepackage{amsopn}
\usepackage{latexsym}
\usepackage{graphicx}        
\usepackage{mathrsfs}		
\usepackage{psfrag}
\usepackage{algorithmic}
\usepackage{algorithm}
\usepackage{color}
\usepackage{verbatim}
\usepackage{empheq}
\usepackage{url}
\usepackage{enumerate}
\usepackage{pgfplots}
\usepackage{csvsimple}
\pgfplotsset{compat=newest}
\usepackage{tikz}
\usepackage{numprint}
\usepackage{varioref}
\usepackage{url}
\usepackage{booktabs} 

\newtheorem{thm}{Theorem}
\newtheorem{hyp}[thm]{Hypothesis}
\newtheorem{cor}[thm]{Corollary}
\newtheorem{prop}[thm]{Proposition}
\newtheorem{lem}[thm]{Lemma}
\newtheorem{rmk}[thm]{Remark}
\newtheorem{defi}[thm]{Definition}

\newcommand{\transpose}[1]{{#1}^{\top}}
\newcommand{\subgaus}[1]{\left\|#1\right\|_{\psi_2}}
\newcommand{\subexp}[1]{\left\|#1\right\|_{\psi_1}}

\newcommand{\eps}{\varepsilon}

\newcommand{\argmin}{\mathop{\rm argmin}\limits}

\renewcommand{\S}{\mathcal{S}}
  
\newcommand{\PSDcone}[1]{{\mathcal{S}^{#1}_+}}

\begin{document}

\begin{frontmatter}

\title{Random projections of linear and semidefinite problems with linear inequalities}

\author{Pierre-Louis Poirion\corref{cor1}}
\ead{pierre-louis.poirion@riken.jp}
\address{Center for Advanced Intelligence Project, RIKEN, Japan}

\author{Bruno F. Louren\c{c}o}
\ead{bruno@ism.ac.jp}
\address{Department of Statistical Inference and Mathematics, Institute of Statistical Mathematics, Japan}

\author{Akiko Takeda}
\ead{takeda@mist.i.u-tokyo.ac.jp}
\address{Department of Creative Informatics, Graduate School of Information Science and Technology,	University of Tokyo, Japan and Center for Advanced Intelligence Project, RIKEN, Japan}

\cortext[cor1]{Corresponding author}

\begin{abstract}
The Johnson-Lindenstrauss Lemma states that there exist linear maps that project a set of points of a vector space into a space of much lower dimension such that the Euclidean distance between these points is approximately preserved. This lemma has been previously used to prove that we can randomly aggregate, using a random matrix whose entries are drawn from a zero-mean sub-Gaussian distribution, the equality constraints of an Linear Program (LP) while preserving approximately the value of the problem. In this paper we extend these results to the inequality case by introducing a random matrix with non-negative entries that allows to randomly aggregate inequality constraints of an LP while preserving approximately the value of the problem. By duality, the approach we propose  allows to reduce both the number of constraints and the dimension of the problem while obtaining some theoretical guarantees on the optimal value. We will also show an extension of our results to certain semidefinite programming instances.
\end{abstract}


\end{frontmatter}


\section{Introduction}\label{sec:int}
Random matrices are matrices $T \in \mathbb{R}^{k \times m}$ whose entries are drawn from a probability distribution. When the underlying distribution is properly chosen, these matrices can have some very interesting properties: the Johnson-Lindenstrauss Lemma (JLL), \cite{dasgupta,jllemma}, states that, if the entries of $T$ are drawn independently from the standard normal distribution  $\mathcal{N}(0,\frac{1}{k})$, it is possible to project a set of $n$ points of $\mathbb{R}^m$ into a space of dimension $k=O(\frac{\log(n)}{\varepsilon^2})$ while preserving approximately (with $\varepsilon$ precision) the Euclidean distance between these points with arbitrarily high probability (w.a.h.p.).\\

Recently, this result has been exploited in \cite{MOR_RP} to prove that equality constraints of an LP written in standard form with inputs  $c \in \mathbb{R}^{n}$, $A \in \mathbb{R}^{m \times n}$ and $b \in  \mathbb{R}^{m}$, could be randomly aggregated using a random matrix $T$ with $k < m$, into a new LP:
\begin{center}
	\begin{minipage}{\textwidth}
		\begin{minipage}{0.48\textwidth}
			\begin{equation*}
			\left\{\begin{array}{llll}
			\min\limits_x & \transpose{{c}}x  \\
			& Ax = {b} \\
			& x \geq 0
			\end{array}\right.
			\end{equation*}
		\end{minipage}
		\begin{minipage}{0.48\textwidth}
			\begin{equation*}
			\left\{\begin{array}{llll}
			\min\limits_x & \transpose{{c}}x  \\
			& TAx = T{b} \\
			& x \geq 0
			\end{array}\right.
			\end{equation*}
		\end{minipage}
	\end{minipage}
\end{center}
while preserving, approximately, w.a.h.p.~the optimal value of the problem. 
Although this result makes it possible to aggregate equality constraints of the form ``$Ax = b$'', aggregating an \emph{inequality constraint} of the form ``$Ax \leq b$'' is a far subtler issue.

For example, one might be tempted to take a look at the duals of the problems above, which leads to problems of the format $\max \{\transpose{b}y \mid c - \transpose{A}y\geq 0 \}$ and $\max\{\transpose{b}(\transpose{T}y_{T}) \mid c - \transpose{A}{\transpose{T}y_T}\geq 0 \}$. Performing the substitution $z = \transpose{T}y_{T}$ leads to a dual problem with less variables, but the number of inequality constraints remains the same. So this approach cannot be used to aggregate inequality constraints.
Alternatively, one may consider adding slack variables $s\in \mathbb{R}^m_+$ to transform inequality constraints into equality constraints. However, this is unlikely to be efficient in the random projection framework because, although the number of constraints is reduced, the number of variables increases by $m$ in the projected problem. Hence using slack variables is a non-starter and we can not expect to reduce the solving time of the problem.

Indeed, to randomly aggregate a set of inequality constraints ``$Ax \le b$'', we need a random matrix $S$ whose entries $S_{ij}$ are non-negative.  In this paper, we propose the first method that allows to randomly aggregate a set of inequality constraints in an LP. More precisely, let us consider the pair:
\begin{center}
	\begin{minipage}{\textwidth}
		\begin{minipage}{0.48\textwidth}
			\begin{equation}\label{eq:lp}
			\mathcal{P}\left\{\begin{array}{llll}
			\min\limits_x & \transpose{{c}}x  \\
			& Ax \ge {b} \\
			& x \in \mathbb{R}^n
			\end{array}\right.
			\end{equation}
		\end{minipage}
		\begin{minipage}{0.48\textwidth}
			\begin{equation}\label{eq:lpt}
			\mathcal{P}_S\left\{\begin{array}{llll}
			\min\limits_x & \transpose{{c}}x  \\
			& SAx \ge S{b} \\
			& x \in \mathbb{R}^n
			\end{array}\right.
			\end{equation}
		\end{minipage}
	\end{minipage}
\end{center}
with $c \in \mathbb{R}^n,\ A \in \mathbb{R}^{m \times n},\ b \in \mathbb{R}^m$. Here, $S \in \mathbb{R}^{k \times m}$ is a random \textit{iid} matrix such that $S_{ij}=\frac{1}{k}T_{ij}^2$ where $T_{ij}$ is drawn from the standard normal distribution $\mathcal{N}\left(0,1\right)$. Although this looks very similar to the equality case it is actually quite different: the random matrix $S$ does not satisfy the JLL property, and hence, a different analysis should be applied. Notice that since each entry of $S$ is non-negative, $\mathcal{P}_S$ is a relaxation of $\mathcal{P}$; $v(\mathcal{P}_S) \le v(\mathcal{P})$ holds (where $v(\cdot)$ denotes the optimal value of an optimization problem). The difficult part is to prove that there exists $\delta(k) >0$ such that, w.a.h.p., 
$$ v(\mathcal{P}) -\delta(k) \le v(\mathcal{P}_S) \le v(\mathcal{P}),$$
where $\delta(k)$ is a decreasing function of $\varepsilon$ (recall that typically $k=O(\frac{\log(m)}{\varepsilon^2})$), which represents the distortion in distance after projection.

More generally, we will consider a pair of linear optimization problems over a cone $\mathcal{K}$ which is a product of the non-negative orthant and semidefinite cones, i.e., $\mathcal{K}=\mathbb{R}^m_{+} \times \PSDcone{p_1}\times \cdots \times \PSDcone{p_l}$:
\begin{center}
	\begin{minipage}{\textwidth}
		\begin{minipage}{0.48\textwidth}
			\begin{equation}\label{eq:lpg}
			\mathcal{P}^{\mathcal{K}}\left\{\begin{array}{llll}
			\min\limits_x & \transpose{{c}}x  \\
			& Ax- {b} \in \mathcal{K} \\
			& x \in \mathbb{R}^n
			\end{array}\right.
			\end{equation}
		\end{minipage}
		\begin{minipage}{0.48\textwidth}
			\begin{equation}\label{eq:lpgt}
			\mathcal{P}^{\mathcal{K}}_Q\left\{\begin{array}{llll}
			\min\limits_x & \transpose{{c}}x  \\
			& \mathcal{Q}(Ax - {b}) \in \mathcal{Q}(\mathcal{K}) \\
			& x \in \mathbb{R}^n
			\end{array}\right.,
			\end{equation}
		\end{minipage}
	\end{minipage}
\end{center}
where $\mathcal{Q}$ is a linear map such that $\mathcal{K}=\mathbb{R}^m_{+} \times \PSDcone{p_1}\times \cdots \times \PSDcone{p_l}$ is mapped to $\mathbb{R}^k_{+} \times \PSDcone{q_1}\times \cdots \times \PSDcone{q_l}$, where $k <m$ and $q_i < p_i$ for all $i$. We will prove that one can build a random map $\mathcal{Q}$ such that, with arbitrarily high probability, $v(\mathcal{P}^{\mathcal{K}}_Q)$ approximates  $v(\mathcal{P}^{\mathcal{K}})$. 

We now review some related works. Notice that in all the works using a random projection matrix $T$ to reduce the dimension of a problem, $T$  preserves approximately the distances, i.e, $\|Tx\|_2 \approx \|x\|_2$ with high probability, and its entries $T_{ij}$ have all zero expectation.
In numerical linear algebra, random projections are used to compress a matrix into a smaller one where computations can be performed quickly  thereby accelerating the solution of the original problem, e.g., \cite{woodruff2014, drineas2016, derezinski202}.
In optimization, random projections are also referred to as \emph{sketching}, especially if random matrices are used not to reduce the dimension of the problem but to reduce the sample size of some data matrix of the problem. For example in a least-square problem setting, i.e. $\min\limits_{x \in C} \|y-X x\|^2$ where $C\subseteq \mathbb{R}^n$ is a convex set, $X \in \mathbb{R}^{m'\times n}$ is the design matrix and $y$ is the response vector,
the sample size $m'$ is reduced using random projections.
In \cite{pilanci2014,2dnoesy,wang2017,dobriban2019}, random projections are used to reduce the size of some large dimensional least squares problem: the data $(X,Y)$ of the problem is replaced by a lower dimensional sketched-data $(SX,SY)$ where $S$ is a random matrix. 
In \cite{necoara2018} and \cite{berahas2020}, iterative methods using random projections are proposed. 
In \cite{bluhm2019}, an application of random projections to the pure semidefinite programming (SDP) case is considered. The random projection map, $M \mapsto TMT^\top$,  considered in \cite{bluhm2019}, preserves the structure of the semidefinite cone, and hence, in opposition to the LP case, the inequality constraints are not a problem. Later in this paper, we will further discuss the relation between the approach in \cite{bluhm2019} and ours. 

To the best of our knowledge,
this is the first work to randomly aggregate linear inequalities. In addition, one of the main contribution of our paper is to use a \emph{non-negative} random matrix for sketching. Indeed, 
 in all papers we have seen using random random projections to reduce the size of a problem, the entries of the random matrices used for the aggregation have zero expectation. 
We remark that zero expectation matrices cannot be used to deal with inequalities constraints because such matrix would typically have negative entries (approximately half of the entries would be negative) and it would imply that the positive orthant is not usually mapped into another positive orthant.

This work is divided as follows. 
In Section \ref{sec:2} we will recall some basic facts about concentration inequalities. Then, in Section \ref{sec:3} we will recall some known results and derive some new ones for random matrices. In Section \ref{sec:4} we will give the main result of this paper and in Section \ref{sec:6}, we will restrict the LP case and obtain a simpler bound. Finally in Section \ref{sec:5} we will present some numerical results.

We resume all the notations used in the paper in Table \ref{tab:myfirsttable}.
{\renewcommand{\arraystretch}{1.15}\begin{table}[tb]
		\centering
		\begin{tabular}{ | c | l | }
			\hline
			Notation & Convention  \\
			\hline
			$\subgaus{X}$ & the sub-Gaussian norm of a random vector \\
			\hline
			$\subexp{X}$ & the sub-exponential norm of a random vector\\
			\hline
			$\mathcal{C}_1$ & absolute constant \\
			\hline
			$\S^p$ & the $p\times p$ real symmetric matrices\\
			\hline
			$\PSDcone{p}$ & the cone of $p\times p$ real positive semidefinite matrices\\
			\hline
			$x^\top y$ & the Euclidean scalar product between $x$ and $y$\\
			\hline
			$\langle M_1, M_2 \rangle_F$ & the Frobenius scalar product of matrices $M_1$ and $M_2$ \\  
			\hline
			$\|M\|_F$, $\|M\|_*$ & respectively the Frobenius norm and the nuclear norm of matrix $M$ \\
			\hline
			$\|M\|_i$ & the induced $\| \cdot \|_i$ for matrix $M$ for $i=1,2$: $\|M\|_i=\max\limits_{\|x\|_i=1}\|Mx\|_i$\\
			\hline
			$|x|$, $x \in \mathbb{R}^n$ & vector whose components are the absolute value of those of $x$ \\
			\hline
			$v(P)$ & optimal value of the optimization problem $(P)$ \\
			\hline 
			$A \circ B $ & Hadamard product of matrices $A$ and $B$ \\
			\hline
			$D(A) $ & matrix whose diagonal is the vector $a$   \\
			\hline
			$D^{-1}(a) $ & diagonal of matrix $A$  \\
			\hline
			$\mathbf{1}$ & the all one vector  \\
			\hline
			
		\end{tabular}
		\caption{Notational conventions}
		\label{tab:myfirsttable}
\end{table}}
\section{Concentration inequalities}\label{sec:2}

In this section we recall some basic facts about concentration inequalities.

\begin{defi}[Sub-Gaussian random variables]
	Let $X$ be a zero mean random variable such that there exists $K>0$ such that for all $t>0$,
	\begin{equation}\label{eq:subgauss}
	P(|X|>t)\le 2\exp\left(-\frac{t^2}{K^2}\right).
	\end{equation}
	Then $X$ is said to be \emph{sub-Gaussian}. The \emph{sub-Gaussian norm} of $X$ is defined to be the smallest $K$ satisfying \eqref{eq:subgauss} and is denoted by $\subgaus{X}$.
\end{defi}

\begin{rmk}\label{rem:2}
	A classical example (\cite[Examples 2.5.8]{vershyninbook}) of sub-Gaussian random variable is a Gaussian random variable $X\sim N(0,\sigma^2)$ with $\subgaus{X} \le 2\sigma$.
\end{rmk}

\begin{defi}[Sub-exponential random variables]
	Let $X$ be a zero mean random variable and let $K>0$ such that for all $t>0$,
	\begin{equation}\label{eq:subexp}
	P(|X|>t)\le 2\exp\left(-\frac{t}{K}\right).
	\end{equation}
	Then $X$ is said to be \emph{sub-exponential}. The \emph{sub-exponential norm} of $X$ is defined to be the smallest $K$ satisfying \eqref{eq:subexp} and is denoted by $\subexp{X}$.
\end{defi}

Sub-Gaussian and sub-exponential random variables are closely related, as we can see that any sub-Gaussian random variable is also sub-exponential (only large values of $t$ are relevant).  Furthermore, it turns out that the product of two sub-Gaussian random variables is sub-exponential.

\begin{lem}[\phantom{}{\cite[Lemma 2.7.7]{vershyninbook}}]\label{lem:1}
	Let $X$ and $Y$ be sub-Gaussian random variables, then XY is sub-exponential, furthermore
	$$\subexp{XY} \le \subgaus{X} \subgaus{Y}.$$
	Also if $Z$ is sub-exponential, then $\subexp{Z - E(Z)} \le  2\subexp{Z}$.
\end{lem}
Next we recall the Bernstein inequality whose proof can be found in \cite[Theorem 2.8.2]{vershyninbook}.
\begin{prop}[Bernstein inequality]
	Let $Y_1,...,Y_N$ be independent, mean zero, sub-exponential random variables. Then, for every $t \ge 0$, we have
	\begin{equation}\label{eq:bern}
	P\left(|\sum\limits_{i=1}^N Y_i| \ge t\right) \le 2\exp\left(-\mathcal{C}_1\min\left(\frac{t^2}{\sum \subexp{Y_i}^2}, \frac{t}{\max \subexp{Y_i} }\right) \right),
	\end{equation}
	where $\mathcal{C}_1>0$ is an absolute constant.
\end{prop}

We now recall the notion of $\hat \varepsilon$-net:
\begin{defi}\label{def:epsnet}
	Given a subset $K \subset \mathbb{R}^m$ and $\hat \varepsilon >0$, we say that a subset $\mathcal{N}$ is an $\hat \varepsilon$-net of $K$ if every point of $K$ is within $\hat \varepsilon$ of a point in $\mathcal{N}$, i.e.,
	$$\forall x \in K,\ \exists y \in \mathcal{N}\ s.t.\ \|x-y\|_2 \le \hat \varepsilon. $$ 
\end{defi}
\begin{rmk}\label{rem:1}
	We can find a $\hat \varepsilon$-net of the $m$-Euclidean ball of size $\left(\frac{2}{\hat \varepsilon}+1\right)^m$ (c.f. \cite[Corollary 4.2.13]{vershyninbook}).
\end{rmk} 
In practice $\hat \varepsilon$-nets can be used to bound the operator norm, $\|M\|_2$, of a matrix $M$:
\begin{lem}[\phantom{}{\cite[Lemma 4.4.1, Exercise 4.4.3]{vershyninbook}}]\label{lem:epsnet}
	Let $M\in \mathbb{R}^{p \times q}$, then for any $\hat \varepsilon$-net $\mathcal{N}$ ($\hat \varepsilon<1$) of the unit sphere $S^{q-1}$ we have
	$$\sup\limits_{x\in \mathcal{N}}\|Mx\|_2 \le \|M\|_2 \le \frac{1}{1-\hat \varepsilon} \sup\limits_{x\in \mathcal{N}}\|Mx\|_2.$$ 
	Furthermore, for any $\hat \varepsilon$-net $\mathcal{N'}$ of $S^{p-1}$ (with $\hat \varepsilon<1/2$), we have
	$$\sup\limits_{x\in \mathcal{N},y \in \mathcal{N'}}\langle Mx,y \rangle \le \|M\|_2 \le \frac{1}{1-2\hat \varepsilon} \sup\limits_{x\in \mathcal{N},y \in \mathcal{N'}}\langle Mx,y \rangle.$$
	Moreover if $p=q$ and $M$ is symmetric, we have 
	$$\sup\limits_{x\in \mathcal{N}}|\langle Mx,x \rangle |\le \|M\|_2 \le \frac{1}{1-2\hat \varepsilon} \sup\limits_{x\in \mathcal{N}}|\langle Mx,x \rangle |.$$
\end{lem}

\section{Properties of random projection matrices}\label{sec:3}

In this section we recall the famous Johnson-Lindenstrauss lemma, which is generalized to sub-Gaussian distribution, and derive some new concentration properties for random matrices. 

\begin{lem}[Johnson-Lindentrauss Lemma (JLL) \cite{achlioptas}]\label{lem:jll0}
	Let $\mathcal{Z}$ be a set of $h$ points in $\mathbb{R}^l$ and let $G$ be a $k \times l$ random matrix whose entries are independent $\mathcal{N}(0,1)$ random variables,  let $0< \varepsilon < 1$. Then with probability $1-2h\exp(-k/2(\varepsilon^2/2-\varepsilon^3/3))$, we have that for all $z_i,z_j \in \mathcal{Z}$
	\begin{equation}\label{eq:jll}
	(1-\varepsilon)\|z_i-z_j\|^2_2 \le \frac{1}{\sqrt{k}}\|Gz_i-Gz_j\|^2_2 \le (1+\varepsilon)\|z_i-z_j\|^2_2,
	\end{equation}
\end{lem}
In the Johnson-Lindenstrauss Lemma, the term  $1-2h\exp(-k/2(\varepsilon^2/2-\varepsilon^3/3))$ means that we can choose $k =O\left(\frac{\log(h)}{\varepsilon^2/2-\varepsilon^3/3}\right)$ so  that Equation~\eqref{eq:jll} is satisfied with probability as small as we want. 

The following Lemma, proved in \cite{IPCO_RP}, enumerates some consequences of the JLL.
\begin{lem}[c.f. {\cite[Lemmas 3.1, 3.2, 3.3]{IPCO_RP}}] \label{jll-approx}  
	For $G$ defined as in Lemma~\ref{lem:jll0}, let $T=\frac{1}{\sqrt{k}}G$. 
	Let $0 < \varepsilon < 1$, then we have
	\begin{enumerate}[(i)]
		\item For any $x, y \in \mathbb{R}^l$,
		\[x^\top y - \varepsilon \|x\|\,\|y\| \le (Tx)^\top(Ty) \le x^\top y + \varepsilon \|x\|\,\|y\|\]
		with probability at least $1 - 4\exp(-k/2(\varepsilon^2/2-\varepsilon^3/3))$.
		\item For any $x\in \mathbb{R}^l$ and $A \in  \mathbb{R}^{p \times l}$  whose $i$th row is denoted by $A_i$,
		\[ 
		Ax - \varepsilon \|x\|  \begin{bmatrix}\|A_1\|_2 \\ \ldots \\ \|A_p\|_2\end{bmatrix}  \le 
		AT^{\top}Tx \le Ax + \varepsilon \|x\|  \begin{bmatrix}\|A_1\|_2 \\ \ldots \\ \|A_p\|_2\end{bmatrix} 
		\]
		with probability at least $1 - 4p\exp(-k/2(\varepsilon^2/2-\varepsilon^3/3))$.
		\item For any two vectors $x, y\in \mathbb{R}^l$ and a square matrix $Q \in  \mathbb{R}^{l \times l}$, then with probability at least $1 - 8r\exp(-k/2(\varepsilon^2/2-\varepsilon^3/3))$, we have
		\[
		x^{\top} Q y - 3\varepsilon \|x\|\, \|y\|\, \|Q\|_F \le x^{\top}T^{\top}TQT^{\top}Ty \le x^{\top} Q y  + 3\varepsilon \|x\|\, \|y\|\, \|Q\|_F,\]
		
		where $r$ is the rank of $Q$.
	\end{enumerate}
\end{lem}
A consequence of Lemma \ref{jll-approx} is that the random mapping $\mathbb{R}^{m \times m} \mapsto \mathbb{R}^{k \times k}:\ M \mapsto TMT^\top$ ``almost'' preserves the Frobenius norm, $\|M\|_{F}$, of $M$.
\begin{lem}\label{lem:matrixapprox}
	Let $G$ be defined as in Lemma \ref{lem:jll0} and let $T=\frac{1}{\sqrt{k}}G$. Let $0 < \varepsilon < 1$. Then for any $A,B \in \mathbb{R}^{l \times l}$, we have that with probability at least $1 - 8r_1r_2\exp(-k/2(\varepsilon^2/2-\varepsilon^3/3))$,
	\begin{equation}\label{eq:matrixscalarproduct}
	\left|\langle A, B \rangle_F - \left\langle TAT^\top, TBT^\top \right\rangle_F\right| \le 3\varepsilon \|A\|_F \|B\|_*,
	\end{equation}
	where $r_1,r_2$ are the ranks of $A$ and $B$, respectively. 
\end{lem} 
\begin{proof}
Assume first that $B$ has rank one, then there exists unit vectors $x,y \in \mathbb{R}^l$ and $\sigma >0$ such that $B=\sigma x y^\top$.
Then by Lemma \ref{jll-approx}($iii$), we have with probability at least $1 - 8r_1\exp(-k/2(\varepsilon^2/2-\varepsilon^3/3))$, 
\[	\sigma x^{\top} A y - 3\varepsilon \sigma\|x\|\, \|y\|\, \|A\|_F \le \sigma x^{\top}T^{\top}TAT^{\top}Ty \le  \sigma x^{\top} A y  + 3\varepsilon \sigma \|x\|\, \|y\|\, \|A\|_F.\]
Since \[ \sigma x^{\top} A y = \langle \sigma x y^\top, A \rangle_F = \langle B, A \rangle_F\]  \[\sigma x^{\top}T^{\top}TAT^{\top}Ty =  \langle \sigma Tx y^\top T^\top, TAT^\top \rangle_F = \langle TBT^\top, TAT^\top \rangle_F\] and $\sigma=\|B\|_*$, this proves the Lemma in the rank one case.

For the general case, we write, using the singular value decomposition of $B$,
$$B = \sum\limits_{i=1}^{r_2} \sigma_i x_i y_i^\top,$$
where $x_i,y_i \in \mathbb{R}^l$ are unit vectors and $\sigma_i >0$. 
We conclude by an union bound on $i\in \{1,\cdots,r_2\}$ (we use Lemma \ref{jll-approx}($iii$) for all $x_i,y_i$), using the linearity of the scalar product and the fact that $\|B\|_*= \sum\limits_{i=1}^{r_2} \sigma_i$. 
\end{proof}
Notice that in the above lemma, the nuclear norm of $A$ or $B$ should be taken into account in the approximation error. In \cite{bluhm2019}, it has been proven that such random mapping cannot preserve the Frobenius norm of a matrix in a similar fashion as the JLL (hence the error cannot be written as $O(\varepsilon \|A\|_F \|B\|_F)$). The approximation we obtain is tighter than the one obtain in \cite{bluhm2019} as our error is $O(\varepsilon \|A\|_F \|B\|_*)$ instead of $O(\varepsilon \|A\|_* \|B\|_*)$.

In the next Lemma, we prove a concentration result for random Gaussian matrices.
\begin{lem}\label{Zhang}
	Let $a \in \mathbb{R}^m_{++}$ and let $U$ be the random $k \times m$ matrix such that its $j$-th column is a random vector drawn, independently from the other columns, from the $\mathcal{N}(0,a_jI_k)$ distribution. Then
	for any $0 < \varepsilon \le 1$, $0 < \delta <\frac{1}{2}$  if 
	$$m \ge \frac{2^8}{\mathcal{C}_1\varepsilon^2}(3k-\ln(\delta)), $$
	then with probability at least $1 - 2\delta$, we have
	$$\left\|\frac{1}{\|a\|_1}UU^{\top} - I_k\right\|_2 \le \frac{\max a_i}{\min a_i}  \varepsilon.$$	
\end{lem}
\begin{proof}
Let us denote, for all $j\le m$, by $A_j$ the $j$th column of $U$ multiplied by $\frac{1}{\sqrt{\|a\|_1}}$. Let us consider a $\frac{1}{4}$-net, $\mathcal{N}$, of the unit sphere $S^{k-1}$ such that $|\mathcal{N}| \le 9^k$ (c.f. Remark \ref{rem:1}).\\
Using Lemma \ref{lem:epsnet} with $\hat \varepsilon=\frac{1}{4}$ and the fact that the matrix $\frac{1}{\|a\|_1}UU^{\top} - I_k$ is symmetric, we deduce that
\begin{equation}\label{eq:lem_zhang:1}
\left\|\frac{1}{\|a\|_1}UU^{\top} - I_k \right\|_2 \le 2  \sup\limits_{x\in \mathcal{N}}\left| \left<\frac{1}{\|a\|_1}UU^{\top}x-x,x \right> \right| = 2 \sup\limits_{x\in \mathcal{N}}  \left| \left\|\frac{1}{\sqrt{\|a\|_1}}U^{\top} x\right\|^2_2 -1\right|.
\end{equation}
Let $x \in S^{k-1}$, we can express $\left\|\frac{1}{\sqrt{\|a\|_1}}U^{\top} x\right\|^2_2$ as a sum of independent random variables:
\begin{equation}\label{eq:1}
\left\|\frac{1}{\sqrt{\|a\|_1}}U^{\top} x\right\|^2_2 = \sum\limits_{i=1}^m \langle A_i ,x \rangle ^2,
\end{equation}
where the $A_i$ are independent sub-Gaussian vectors distributed under the $\mathcal{N}(0,\frac{a_i}{\|a\|_1}I_k)$ distribution for every $i$. Thus, by Remark \ref{rem:2}, $X_i=\langle A_i ,x \rangle$	are independent sub-Gaussian  random variables with $$E(X_i^2)=\frac{a_i}{\|a\|_1}  \mbox{ and } \subgaus{X_i} \le 2 \sqrt{\frac{a_i}{\|a\|_1}}.$$
Therefore, $Y_i = X_i^2-\frac{a_i}{\|a\|_1}$ are independent zero means, sub-exponential random variables and by Lemma~
\ref{lem:1} we have 
$$ \subexp{X_i^2} \le \subgaus{X_i}\subgaus{X_i} \le 2^2\frac{a_i}{\|a\|_1},$$
hence by Lemma \ref{lem:1} again,
$$\subexp{X_i^2-\frac{a_i}{\|a\|_1}}\le 2\subexp{X_i^2}  \le 2^3\frac{a_i}{\|a\|_1}. $$
Notice that $\sum\limits_{i=1}^m \frac{a_i}{\|a\|_1}=1$.
Using Bernstein inequality \eqref{eq:bern} and \eqref{eq:1},  we obtain
\begin{align}
P\left(\left| \left\|\frac{1}{\sqrt{\|a\|_1}}U^{\top} x\right\|^2_2 -1\right| \ge \frac{1}{2}\frac{\max a_i}{\min a_i}\varepsilon\right) &= P\left( \left|\sum\limits_{i=1}^m(X_i^2-\frac{a_i}{\|a\|_1})\right| \ge \frac{1}{2}\frac{\max a_i}{\min a_i} \varepsilon\right) \notag\\
& \le 2\exp\left(-\mathcal{C}_1\min\left(\frac{\left(\frac{\max a_i}{\min a_i}\varepsilon\right)^2}{4\sum\limits_{i=1}^m \left(2^3\frac{a_i}{\|a\|_1}\right)^2}, \frac{\frac{\max a_i}{\min a_i}\varepsilon}{2\max\limits_i \left(2^3\frac{a_i}{\|a\|_1}\right) }\right) \right). \label{eq:puai}
\end{align}
In order to bound \eqref{eq:puai}, we use  the following inequality with $a \in \mathbb{R}^m_{++}$:
\begin{align}
\sum\limits_{i=1}^m \left(2^3\frac{a_i}{\|a\|_1}\right)^2 &= 2^6 \frac{\sum\limits_i a_i^2}{\sum\limits_i a_i^2 + 2\sum\limits_{i < j} a_ia_j} \notag \\
&\le 2^6\frac{m (\max a_i)^2}{m (\min a_i)^2+(m-1)m(\min a_i)^2} \notag \\
& = 2^6 \frac{(\max a_i)^2}{m(\min a_i)^2}. \label{eq:cma}
\end{align}
Then, using \eqref{eq:cma}, the fact that $\max\limits_i \left(2^3\frac{a_i}{\|a\|_1}\right) \le 2^3 \frac{\max\limits_i a_i}{m \min\limits_i a_i} $
we plug all those bounds in \eqref{eq:puai} to obtain
\begin{align*}
P\left(\left| \left\|\frac{1}{\sqrt{\|a\|_1}}U^{\top} x\right\|^2_2 -1 \right| \ge \frac{1}{2}\frac{\max a_i}{\min a_i}\varepsilon \right) &\le 2 \exp\left(-\mathcal{C}_1\min\left(\frac{1}{2^8}\varepsilon^2m, \frac{1}{2^4}\varepsilon m\right) \right) \\
&\le 2\exp\left(-\frac{\mathcal{C}_1}{2^8}\varepsilon^2 m\right).
\end{align*}
Now using an union bound on the set $\mathcal{N}$, we have that
\begin{align*}
P\left(\left| \sup\limits_{x\in \mathcal{N}}\left\|\frac{1}{\sqrt{\|a\|_1}}U^{\top} x\right\|^2_2 -1\right|\ge  \frac{1}{2}\frac{\max a_i}{\min a_i}\varepsilon \right) & \le 2* 9^k \exp\left(-\frac{\mathcal{C}_1}{2^8}\varepsilon^2 m\right) \\
&   \le 2\exp \left(3k-\frac{\mathcal{C}_1}{2^8}\varepsilon^2 m \right).  
\end{align*}
From \eqref{eq:lem_zhang:1} we have
$$P\left( \left\|\frac{1}{\|a\|_1}UU^{\top} - I_k \right\|_2 \ge \frac{\max a_i}{\min a_i} \varepsilon \right) \le P\left(\left| \sup\limits_{x\in \mathcal{N}}\left\|\frac{1}{\sqrt{\|a\|_1}}U^{\top} x\right\|^2_2 -1\right|\ge  \frac{1}{2}\frac{\max a_i}{\min a_i}\varepsilon \right).$$
Hence, $$P\left( \left\|\frac{1}{\|a\|_1}UU^{\top} - I_k \right\|_2 \ge \frac{\max a_i}{\min a_i}\varepsilon \right)  \le 2\exp \left(3k-\frac{\mathcal{C}_1}{2^8}\varepsilon^2 m \right) \le 2\exp(\ln(\delta))=2\delta .$$  
\end{proof}
In the above Lemma, we proved that $\frac{1}{\|a\|_1}UU^{\top}$ concentrates around its expectation, $\mathbb{E} \left(\frac{1}{\|a\|_1}UU^{\top} \right) = I_k$. This is a generalization of a result proved in \cite{zhang13} about concentration of Gaussian random matrices. 
\begin{cor}\label{cor:2}
	Let $a \in \mathbb{R}^m_{++}$ and let $U$ be a random Gaussian $k \times m$ matrix such that its $j$-th column is  drawn, independently from the $\mathcal{N}(0,I_k)$ distribution, then
	for any $0 < \varepsilon \le 1$, $0 < \delta <\frac{1}{2}$ if 
	$$m \ge \frac{2^8}{\mathcal{C}_1\varepsilon^2}(3k-\ln(\delta)),$$
	then with probability at least $1 - 2\delta$, we have
	$$\left\|UD(a)U^{\top} - \|a\|_1I_k\right\|_2 \le \frac{\max a_i}{\min a_i} {\|a\|_1} \varepsilon,$$
	where $D(a)$ denotes the $m\times m$ diagonal matrix built from the vector $a$. 
\end{cor}
\begin{proof}
Let $U'$ be the matrix $U'=\sqrt{D(a)}U$. The $j$th column of $U'$ follows the $\mathcal{N}(0,a_jI_k)$ distribution, hence, by Lemma~\ref{Zhang}, with probability at least $1 - 2\delta$, we have
$$\left\|\frac{1}{\|a\|_1}U'U'^{\top} - I_k\right\|_2 \le \frac{\max a_i}{\min a_i}  \varepsilon,$$
which ends the proofs after multiplying both sides by $\|a\|_1$. 
\end{proof}

\section{The projected problem}\label{sec:4}
In this section, we will analyze randomly projected versions of the conic optimization problem discussed in Section~\ref{sec:int}. However, there are a number of technical assumptions we need to impose and their degree of restrictiveness vary.
Strictly speaking, the problems for which our results are valid must have the following shape
\begin{equation}\label{eq:lpg2}
\mathcal{P}^{\mathcal{K}}\left\{\begin{array}{llll}
\min\limits_x & \transpose{{c}}x  \\
& Ax- {b} \in \mathcal{K} \\
& Bx -d \in \mathcal{K}' \\
& x \in \mathbb{R}^n,
\end{array}\right.
\end{equation}
where $\mathcal{K} = \mathbb{R}^m_{+} \times \PSDcone{p_1}\times \cdots \times \PSDcone{p_l}$ and $\mathcal{K}'$ is some arbitrary self-dual cone\footnote{$\mathcal{K}'$ is self-dual if and only if 
	$\mathcal{K}' = \{u \mid \transpose{{u}}v \geq 0, \forall v \in \mathcal{K}' \}$. }.
In the remaining of the paper, an element $y\in \mathcal{K}$ will be denoted by $y=(y_0,M_1,\cdots,M_l)$. 

\begin{hyp}\label{hyp:0}
	We make the following assumptions on \eqref{eq:lpg2}.
	\begin{enumerate}[$(i)$]
		\item $\mathcal{K}'$ is a self-dual cone and the set $\{x\ |\ Bx -d \in \mathcal{K}'\}$ is non-empty such that $\min\limits_{ Bx -d \in \mathcal{K}'} c^\top x$ has finite value (for example, if $c\ge 0$ we can consider the set $\{x \in \mathbb{R}^n | x\ge 0  \}$). \label{hyp:1}
		\item For all $i\in \{1,\cdots,n\}$, $c_i\neq 0$. \label{hyp:2}
		\item All the $l,m,p_1,\cdots,p_l$ are all big-O of $n$. \label{hyp:3}
		\item The optimal value of \eqref{eq:lpg2} and its dual
		coincide. In addition, both problems have optimal solutions. \label{hyp:4}
		\item There exists an optimal solution $(y^*,\lambda^*)$ to the dual problem of \eqref{eq:lpg2}, where $y^*=(y_0^*,M_1^*,\cdots,M_l^*)$, such that ${y^*_0}_j \neq 0$,
		for all $j \in \{1,\cdots,m\}$.  \label{hyp:5}
		\item The optimal value, $v(\mathcal{P}^{\mathcal{K}})$, of $\mathcal{P}^{\mathcal{K}}$ is non-zero. \label{hyp:6}
	\end{enumerate}
\end{hyp}

\begin{rmk}
	Notice that (\ref{hyp:2}), (\ref{hyp:4}) hold generically (c.f. \cite{PT01,dur2017}) and (\ref{hyp:6}) also holds generically. 
	Furthermore, 
	regarding item $(\ref{hyp:5})$, in practice we can always consider a point $\tilde{y}^*$ in a neighborhood of $y^*$, instead of $y^*$, such that $|b^\top y^* -b^\top \tilde{y}^*|\le \eps$.	
	As for  (\ref{hyp:6}) it is not needed if we consider error bounds in absolute value instead of relative value with respect to $v(\mathcal{P}^{\mathcal{K}})$. In any case, we can perturb the vector $c$ by a random quantity $\delta c$ of small variance to ensure that the hypothesis holds.\\
	As for (\ref{hyp:1}) it is indeed necessary to prove that the projected problem is bounded.
\end{rmk}

Let us consider the following random map:
\[\mathcal{Q}:\ \mathbb{R}^m \times \S^{p_1}  \times \cdots \times \S^{p_l}  \mapsto \mathbb{R}^k \times \S^{q_1}  \times \cdots \times \S^{q_l} .\]
We have
\begin{equation}\label{eq:def_Q}
\mathcal{Q}((y_0,M_1,\cdots,M_l))=(Sy_0,Q_1(M_1),\cdots,Q_l(M_l)),
\end{equation}
where $S \in \mathbb{R}^{k \times m}$ is a random \textit{iid} matrix such that  $S=T\circ T$ where $\circ$ denotes the Hadamard product and where $T_{ij}$ is drawn from the  normal distribution $\mathcal{N}(0,\frac{1}{k})$. Furthermore, we have
\begin{equation}\label{eq:def_Qi}
Q_i(M_i)={T^{(i)}} M_i {T^{(i)}}^\top,
\end{equation}
where for all $i \in \{1,\cdots,l\}$, ${T^{(i)}} \in \mathbb{R}^{q_i \times p_i}$ are random \textit{iid} matrices such that each entry is drawn independently from $\mathcal{N}\left(0,\frac{1}{q_i}\right)$. \\
Notice that $\mathcal{Q}$ and $\mathcal{K}=\mathbb{R}^m_{+} \times \PSDcone{p_1}\times \cdots \times \PSDcone{p_l}$ satisfy \[\mathcal{Q}(\mathcal{K})=\mathbb{R}^k_{+} \times \PSDcone{q_1}\times \cdots \times \PSDcone{q_l},\] which is also the product of a non-negative orthant and positive semidefinite cones.

\begin{rmk}
	  While it is true that LP \eqref{eq:lp} can be written as an SDP,  the method proposed in \cite{bluhm2019} would not be efficient if applied to \eqref{eq:lp}. This is because the resulting projected problem is an SDP which would not be reducible to an LP again.
	Hence, 
	we would need to solve the projected problem as an SDP. 
\end{rmk}

We consider the following projected problem:
\begin{equation}\label{eq:lpgt2}
\mathcal{P}^{\mathcal{K}}_Q\left\{\begin{array}{llll}
\min\limits_x & \transpose{{c}}x  \\
& \mathcal{Q}(A)x - \mathcal{Q}({b}) \in \mathcal{Q}(\mathcal{K}) \\
& Bx -d \in \mathcal{K}' \\
& x \in \mathbb{R}^n,
\end{array}\right.
\end{equation}
where $\mathcal{Q}(A)$ is the matrix whose columns are $\mathcal{Q}(A_i)$ where $A_i$ is the $i$th column of $A$. Notice that $\mathcal{P}^{\mathcal{K}}_Q$ is a relaxation of $\mathcal{P}^{\mathcal{K}}$, hence 
$$v(\mathcal{P}^{\mathcal{K}}_Q) \le v(\mathcal{P}^{\mathcal{K}}).$$
We now derive a lower bound for the value of $\mathcal{P}^{\mathcal{K}}_Q$.

Let us consider the duals, $\mathcal{D}^{\mathcal{K}}$ of \eqref{eq:lpg2} and $\mathcal{D}^{\mathcal{K}}_\mathcal{Q}$ of \eqref{eq:lpgt2}:
\begin{equation}
\left.\begin{array}{llll}
\max\limits_{y,\lambda} & \transpose{b} y +\transpose{d} \lambda && \\
& A^\top y + B^\top \lambda &=& c \\
& y \in \mathcal{K}, \lambda \in \mathcal{K}'
\end{array}\right\} \mathcal{D}^{\mathcal{K}}\label{eq:d}
\end{equation}
\begin{equation}
\left.\begin{array}{llll}
\max\limits_{z,\lambda} & \transpose{(\mathcal{Q}(b))}z  +\transpose{d} \lambda && \\
& A^\top \mathcal{Q}^\top(z) +B^\top \lambda  &=& c \\
& z \in \mathcal{Q}(\mathcal{K}), \lambda \in \mathcal{K}'
\end{array}\right\} \mathcal{D}^{\mathcal{K}}_\mathcal{Q}\label{eq:dt}
\end{equation}
where $\mathcal{Q}^\top$ denotes the dual of the map $\mathcal{Q}$. Let $(y^*,\lambda^*) \in \mathcal{K}\times \mathcal{K}'$ be an optimal solution of \eqref{eq:d}. We consider the following ``approximated'' projected solution, $(z_Q,\lambda^*)$, where
\begin{equation}\label{eq:zQ}
z_Q :=  \mathcal{Q}(y^*)  \in \mathcal{Q}(\mathcal{K}).
\end{equation}

We will now prove that $(z_Q,\lambda^*)$ is ``almost'' feasible for $\eqref{eq:dt}$.
Let us consider the modified dual problem:
\begin{equation} 
\left.\begin{array}{llll}
\max\limits_{z,\lambda} & \transpose{({\mathcal{Q}}(b))}z +\transpose{d} \lambda  && \\
& A^\top \mathcal{Q}^\top(z) +B^\top \lambda  &=& c + A^\top(\mathcal{Q}^\top(z_Q) - y^*)\\
& z \in \mathcal{Q}(\mathcal{K}), \lambda \in \mathcal{K}'
\end{array}\right\} \mathcal{D}^\eps_Q\label{eq:dteps}
\end{equation}
Notice that by definition of $\mathcal{D}^\eps_Q$, $(z_Q,\lambda)$ is a feasible solution for \eqref{eq:dteps}. We will now prove that $\mathcal{Q}^\top(z_Q)=\mathcal{Q}^\top(\mathcal{Q}(y^*))$ is ``close'' to $y^*$, which will be enough to obtain a lower bound on $v(\mathcal{P}_\mathcal{Q}^\mathcal{K})$ in Theorem \ref{thm:1}.
Let $E \in \mathbb{R}^n$ denote the ``error'': 
\begin{equation}\label{eq:E}
E:=A^\top(\mathcal{Q}^\top(z_Q) - y^*)= A^\top \left(\mathcal{Q}^\top(\mathcal{Q}(y^*)) -y^*\right).
\end{equation}

\subsection{Bounding the error $E$}
Let us write $y^*=(y^*_0,M^*_1,\cdots,M^*_l)\in \mathbb{R}^m \times \S^{p_1} \times \cdots \times \S^{p_l} $. We have that $$\mathcal{Q}^\top(\mathcal{Q}(y^*))=\left(S^\top Sy^*_0,{T^{(1)}}^\top {T^{(1)}} M_1^* {T^{(1)}}^\top {T^{(1)}},\cdots , {T^{(l)}}^\top {T^{(l)}} M_l^* {T^{(l)}}^\top {T^{(l)}}\right),$$ hence
\begin{equation}\label{eq:2}
\mathcal{Q}^\top(z_Q) - y^*= \begin{pmatrix}
S^\top Sy^*_0 - y^*_0 \\
{T^{(1)}}^\top {T^{(1)}} M_1^* {T^{(1)}}^\top {T^{(1)}} - M_1^* \\
\vdots \\
{T^{(l)}}^\top {T^{(l)}} M_l^* {T^{(l)}}^\top {T^{(l)}} - M_l^*
\end{pmatrix}.
\end{equation}
Let 
\begin{equation}\label{eq:5}
A^\top= \begin{pmatrix}
A^{(0)} & A^{(1)} & \cdots & A^{(l)}
\end{pmatrix}^\top
\end{equation}
be the column decomposition of $A^\top$ such that for all $y=(y_0,M_1,\cdots,M_l)$, and hence, 
\begin{equation}\label{eq:decomp}
A^\top y = {A^{(0)}}^\top y_0 + \sum\limits_{i=1}^l {A^{(i)}}^\top M_i,
\end{equation}
where $A^{(0)} \in \mathbb{R}^{m \times n}$ and $A^{(i)} \in \mathbb{R}^{p_i^2 \times n }$ for $1\le i \le l$. Here we use the notation ${A^{(i)}}^\top M_i$ to denote the vector in $\mathbb{R}^n$ whose $j$th component is given by $ \langle A^{(i)}_j, M_i \rangle $, where the $p_i \times p_i$ matrix $A^{(i)}_j$ is seen as the $j$th column of $A^{(i)}$.\\
Using \eqref{eq:decomp}, $E$ in \eqref{eq:E} can be written as
\begin{equation}\label{eq:error2}
E= {A^{(0)}}^\top (S^\top Sy^*_0 - y^*_0) + \sum\limits_{i=1}^l {A^{(i)}}^\top \left({T^{(i)}}^\top {T^{(i)}} M_i^* {T^{(i)}}^\top {T^{(i)}} - M_i^*    \right).
\end{equation}
The goal of this subsection is to prove the following proposition.
\begin{prop}\label{prop:2}
	Let $\varepsilon,\delta,m$ be such that  $0 <\delta < \frac{1}{8} $,  $0 < \varepsilon < 1$ and
	$m \ge \frac{2^8}{\mathcal{C}_1\varepsilon^2}(3k+ \ln(n)-\ln(\delta)) $.
	Then, 	with probability at least $1-8\delta - (8m^2 +4m)(n+1)\exp(-k/2(\varepsilon^2/2-\varepsilon^3/3))-\sum\limits_{i=1}^l 8p_i^2(n+1) \exp({-q_i/2(\varepsilon^2/2-\varepsilon^3/3)})$,
	\begin{gather}
	|E| \le \varepsilon \alpha(y^*_0,A^{(0)}) \begin{pmatrix}
	\|A^{(0)}_1\|_2 \| y^*_0\|_2 \\
	\vdots \\
	\|A^{(0)}_n\|_2 \| y^*_0\|_2
	\end{pmatrix} + 3 \varepsilon\left(\max\limits_{i=1,\cdots,l} \frac{\|M^*_i\|_*}{\|M^*_i\|_F} \right) \sum\limits_{i=1}^l  \|M_i^*\|_F \begin{pmatrix}
	\|A^{(i)}_1\|_F \\
	\vdots \\
	\|A^{(i)}_n\|_F
	\end{pmatrix},\label{eq:7}\\
	|\transpose{(Q(b))}z_Q - b^\top y^*|  \le \varepsilon \alpha(y^*_0,b_0) \|b_0\|_2 \|y^*_0\|_2 + 3 \varepsilon \left(\max\limits_{i=1,\cdots,l} \frac{\|M^*_i\|_*}{\|M^*_i\|_F} \right)\sum\limits_{i=1}^l  \|b_i\|_F\|M_i^*\|_F, \label{eq:8}	
	\end{gather}
	where ${A^{(i)}_j}$ denotes the $j$th column of ${A^{(i)}}$, where $|E|$ is the vector whose components are the absolute value of the components of $E$ and where
	\begin{equation}\label{alphadef}
	\alpha(y^*_0,A^{(0)})  =  16 \left(\max_j\frac{\|A^{(0)}_j\|_1}{k\|A^{(0)}_j\|_2} \left(\frac{\|y^*_0\|_1}{\|y^*_0\|_2}\frac{\max|{y^*_0}_i|}{\min |{y^*_0}_i|}(1+\varepsilon) +k \right) \right),
	\end{equation}
	\begin{equation}\label{alphadef2}
	\alpha(y^*_0,b_0)  =  16 \left(\frac{\|b_0\|_1}{k\|b_0\|_2} \left(\frac{\|y^*_0\|_1}{\|y^*_0\|_2}\frac{\max|{y^*_0}_i|}{\min |{y^*_0}_i|}(1+\varepsilon) +k \right) \right) 
	\end{equation}
	and $b=(b_0,b_1,\cdots,b_l)$.
	
\end{prop}
We will show the proof of Proposition~\ref{prop:2} later after presenting 
a few preliminary results.
In particular, to obtain a bound on $E$, we will bound each term in the summation in \eqref{eq:error2}.

The first step is to use item~$(iii)$ of Lemma~\ref{jll-approx} for all $1 \le i\le l$ in order to bound the terms
$$ {A^{(i)}}^\top \left({T^{(i)}}^\top {T^{(i)}} M_i^* {T^{(i)}}^\top {T^{(i)}} - M_i^*    \right).$$
Indeed, let us denote by ${A^{(i)}_j}$ the $j$th column of ${A^{(i)}}$. ${A^{(i)}_j}$ is a $p_i \times p_i$ matrix, hence by Lemma \ref{lem:matrixapprox}, for every $\varepsilon \in (0,1)$ we have that with probability at least $1 - 8p_i^2 \exp(-q_i/2(\varepsilon^2/2-\varepsilon^3/3))$,
\begin{equation} \label{eq:2a}
-3 \varepsilon \|M_i^*\|_* \|{A^{(i)}_j}^\top\|_F  \le \left\langle A^{(i)}_j , \left({T^{(i)}}^\top {T^{(i)}} M_i^* {T^{(i)}}^\top {T^{(i)}} - M_i^*\right)\right\rangle_F \le 3 \varepsilon \|M_i^*\|_* \|{A^{(i)}_j}^\top\|_F.
\end{equation}	
By an union bound, considering all the $j \in \{1,\cdots,n\}$, we have that with probability at least $1 - 8p_i^2n \exp(-q_i/2(\varepsilon^2/2-\varepsilon^3/3))$,  
\begin{equation}\label{eq:6}
|{A^{(i)}}^\top( {T^{(i)}}^\top {T^{(i)}} M_i^* {T^{(i)}}^\top {T^{(i)}} - M_i^*) | \le   3 \varepsilon \|M_i^*\|_F \begin{pmatrix}
\|A^{(i)}_1\|_F \\
\vdots \\
\|A^{(i)}_n\|_F
\end{pmatrix}\frac{\|M_i^*\|_*}{\|M_i^*\|_F}.
\end{equation}
Notice that 
Lemma~\ref{lem:jll0} and Lemma~\ref{lem:matrixapprox} imply 
\begin{equation}\label{eq:4}
q_i = O\left( \frac{\log(p_i)}{\varepsilon^2} \right).
\end{equation}
Since $p_i=O(n)$ by Hypothesis~\ref{hyp:0}\eqref{hyp:3}, this ensures that Equation~\eqref{eq:6} holds w.a.h.p.

In order to complete the task of bounding $E$, next we bound the term ${A^{(0)}}^\top (S^\top Sy^*_0 - y^*_0)$. In what follows, $D(\cdot):\mathbb{R}^m \to \S^m$ denotes the function that maps 
a vector $y_0 \in \mathbb{R}^m$ into a $m\times m$ diagonal matrix with $y_0$ in its entries. Then, $D^{-1}(\cdot)$ is the function that maps a matrix to its diagonal vector. We have the following lemma.
\begin{lem}\label{lem:3}
	For any $y_0 \in \mathbb{R}^m_+$, we have that 
	\[
	Sy_0=D^{-1}(TD(y_0)T^\top).
	\]
\end{lem}
\begin{proof}
Let $U=\sqrt{D(y_0)}$. We have $TD(y_0)T^\top=(TU)(TU)^\top$, hence the $i$th term on the diagonal of $TD(y_0)T^\top$ is equal to $\sum\limits_{j=1}^m(TU)_{ij}^2$. Since $U$ is a diagonal matrix: \[U=\begin{pmatrix}
\sqrt{{y_0}_1} & 0 & \cdots & 0 \\
0      & \ddots & & 0 \\
0 & \cdots & 0 & \sqrt{{y_0}_m}
\end{pmatrix},\] we deduce that $(TU)_{ij}=T_{ij}\sqrt{{y_0}_j}$. 
Hence the $i$th term on the diagonal of $TD(y_0)T^\top$ is equal to $\sum\limits_{j=1}^m T_{ij}^2 {y_0}_j=(Sy_0)_i$. 
\end{proof}
We deduce from Lemma~\ref{lem:3} the following corollary.
\begin{cor}\label{cor:1}
	We have that for any $y_0 \in \mathbb{R}^m_+$
	$$S^\top Sy_0= D^{-1}(T^\top D(D^{-1}(TD(y_0)T^\top))T).$$
\end{cor}

\begin{prop}\label{prop:1}
	Let $y_0 \in \mathbb{R}^m$ such that ${y_0}_i\neq 0$ for all $i$, let $0<\delta<\frac{1}{4}$. Assuming that $0 < \varepsilon \le \mathcal{C}^3$, and that 
	$m \ge \frac{2^8}{\mathcal{C}_1\varepsilon^2}(3k-\ln(\delta)) $, we have that 
	$$\|D(Sy_0) - TD({y_0})T^\top\|_2 \le 16\varepsilon \frac{\max |{y_0}_i|}{k \min |{y_0}_i|}\|{y_0}\|_1 $$
	holds with probability at least $1-4\delta$.
\end{prop}
\begin{proof}
We first assume that $y_0 \in \mathbb{R}^m_{++}$. Let $U \in \mathbb{R}^{k \times m}$ be defined by $U= \sqrt{k}T$. By Corollary \ref{cor:2}, we conclude that with at least probability  $1-2\delta$,
\[ 
\left\|UD(y_0)U^\top - \|y_0\|_1I_k \right\|_2 \le \varepsilon \frac{\max {y_0}_i}{ \min{y_0}_i}{\|y_0\|_1}.
\]
Hence we deduce that 
\begin{equation}\label{eq:a}
\left\|TD({y_0})T^\top- \frac{\|{y_0}\|_1}{k}I_k \right\|_2 \le \varepsilon \frac{\max {y_0}_i}{\min{y_0}_i} \frac{\|y_0\|_1}{k}.
\end{equation}
Furthermore, the $ii$-th element of  $D(D^{-1}(TD(y_0)T^\top))$ is given by
$$e_i^\top TD(y_0)T^\top e_i. $$ We deduce from \eqref{eq:a} that for all $i\le k$
$$ \left| e_i^\top TD(y_0)T^\top e_i -   \frac{\|y_0\|_1}{k}\right| \le \varepsilon \frac{\max {y_0}_i}{k \min {y_0}_i}\|y_0\|_1 $$
and, since $D(D^{-1}(TD(y_0)T^\top)) -  \frac{\|y_0\|_1}{k} I_k$ is diagonal, we have that 
\begin{equation}\label{eq:b}
\left\|D(D^{-1}(TD(y_0)T^\top)) -  \frac{\|y_0\|_1}{k} I_k\right\|_2 \le \varepsilon \frac{\max {y_0}_i}{k \min {y_0}_i}\|y_0\|_1.
\end{equation}
By combining \eqref{eq:a}, \eqref{eq:b}, the triangle inequality and Lemma \ref{lem:3}
we conclude that for $y_0 \in \mathbb{R}^m_+$ we have with at least probability  $1-2\delta$,
\begin{equation}\label{eq:c}
\|D(Sy_0) - TD({y_0})T^\top\|_2 \le 2\varepsilon \frac{\max |{y_0}_i|}{k \min |{y_0}_i|}\|{y_0}\|_1.
\end{equation}
For the general case, write $y_0=y_0^+-y_0^-$ where $y_0^+,y_0^-\in \mathbb{R}^m_{++}$ are chosen in the following way
\begin{itemize}
	\item $y_0^+$ is the sum of the positive part of $y$ and $\min(|{y_0}_i|)\mathbf{1}$,
	\item $y_0^-$ is the sum of the negative part of $y$ and $\min(|{y_0}_i|)\mathbf{1}$,
\end{itemize}
where $\mathbf{1} \in \mathbb{R}^m$ is the vector having all entries equal to $1$.
Since $y_0$ has no zero components we have that $\|\min_i(|{y_0}_i|)\mathbf{1}\|_1 \le \|y_0\|_1$, hence,
$\|y_0^+\|_1 \le 2\|y_0\|_1$ and $\|y_0^-\|_1 \le 2\|y_0\|_1$. Furthermore, $\max_i {y_0}_i^+ \le 2\max_j |{y_0}_j|$, $\max_i {y_0}_i^- \le 2\max_j |{y_0}_j|$, $\min_i {y_0}_i^+ \ge \min_j |{y_0}_j|$ and $\min_i {y_0}_i^- \ge \min_j |{y_0}_j|$ hold.     \\
We have 
\begin{align*}
&\|D(S(y_0)) - TD({y_0})T^\top\|_2= \\ &\|(D(S(y^+_0))-D(S(y^-_0))) - (TD({y^+_0})T^\top-TD({y^-_0})T^\top)\|_2.
\end{align*}
Since $y^+_0$ and $y^-_0$ are positive, we conclude from \eqref{eq:c} that with at least probability  $1-4\delta$\footnote{This probability is obtained by an union bound.}:
\begin{align*}
&\|D(S(y_0)) - TD({y_0})T^\top\|_2 \le \\
&\|D(S(y^+_0)) - TD({y^+_0})T^\top\|_2+\|D(S(y^-_0)) - TD({y^-_0})T^\top\|_2 \le \\
&2\varepsilon \left(\frac{\max {y_0}_i^+ }{k \min {y_0}_i^+ }\|y^+_0\|_1 +\frac{\max {y_0}_i^- }{k \min {y_0}_i^- }\|y^-_0\|_1\right) \le 16\varepsilon \frac{\max |{y_0}_i|}{k \min {|y_0}_i|}\|{y_0}\|_1,
\end{align*}
where we also used the relations between $\|y_0^-\|_1,\|y_0^+\|_1$ and $\|y_0\|_1$ as well the relations between $\min(|{y_0}_i|)$, $\max(|{y_0}_i|)$, $\min(|{y_0}^+_i|)$, $\max(|{y_0}^+_i|)$, $\min(|{y_0}^-_i|)$, $\max(|{y_0}^-_i|)$. 
\end{proof}

\begin{prop}\label{prop:3}
	Let $S$ be as in \eqref{eq:def_Q} and let $y^1_0, y^2_0 \in \mathbb{R}^m$ be such that for all $i\in \{1,\cdots,m\}$, ${y^2_0}_i\neq 0$. Assume that $ m\ge  \frac{2^8}{\mathcal{C}_1\varepsilon^2}\left(3k-\ln(\delta)\right)$. Then, with probability at least $1 - 4 \delta - (8m^2+4m)\exp(-k/2(\varepsilon^2/2-\varepsilon^3/3))$, we have
	\[
	|(Sy^1_0)^\top(Sy^2_0) - {y^1_0}^\top y^2_0| \le \varepsilon \alpha(y^1_0, y^2_0 ) \|y^1_0\|_2  \|y^2_0\|_2,
	\]
	where $\alpha(y^1_0, y^2_0 ) = 16 \left(\frac{\|{y^1_0}\|_1}{k\|{y^1_0}\|_2} \left(\frac{\|{y^2_0}\|_1}{\|{y^2_0}\|_2}\frac{\max |{y^2_0}_i|}{\min |{y^2_0}_i|}(1+\varepsilon) +k \right) \right).$
\end{prop}
\begin{proof}
We have that 
$(Sy^1_0)^\top(Sy^2_0)={y^1_0}^\top(S^\top S y^2_0)$, hence by Corollary \ref{cor:1}, we have that 
$$ (Sy^1_0)^\top(Sy^2_0)= {y^1_0}^\top D^{-1}(T^\top D(D^{-1}(TD(y^2_0)T^\top))T).$$
Since $D^{-1}(T^\top D(D^{-1}(TD(y^2_0)T^\top))T)$ is the vector whose $i$th component is $(Te_i)^\top D(D^{-1}(TD(y^2_0)T^\top)) (Te_i)$,
we have that
\begin{equation}\label{eq0}
(Sy^1_0)^\top(Sy^2_0)=\sum\limits_{i=1}^m {y^1_0}_i\left(e_i^\top T^\top TD(y^2_0)T^\top Te_i+  (Te_i)^\top\left(  D(D^{-1}(TD(y^2_0)T^\top)) -  TD(y^2_0)T^\top \right)(Te_i) \right).
\end{equation}
Hence we have that
\begin{align}\label{eq1}
|(Sy^1_0)^\top(Sy^2_0) - {y^1_0}^\top y^2_0 |\le &\left|\sum\limits_{i=1}^m {y^1_0}_ie_i^\top T^\top TD(y^2_0)T^\top Te_i - {y^1_0}^\top y^2_0 \right| \notag \\ 
&+ \left|\sum\limits_{i=1}^m {y^1_0}_i\left( (Te_i)^\top\left(  D(D^{-1}(TD(y^2_0)T^\top)) -  TD(y^2_0)T^\top \right)(Te_i) \right)\right|.
\end{align}	
First let us bound the term $\left|\sum\limits_{i=1}^m {y^1_0}_ie_i^\top T^\top TD(y^2_0)T^\top Te_i - {y^1_0}^\top y^2_0 \right|$ by using
$$\sum\limits_{i=1}^m {y^1_0}_i e_i^\top T^\top TD(y^2_0)T^\top Te_i =\left\langle D(y^1_0),  T^\top TD(y^2_0)T^\top T \right\rangle_F = \left\langle T D(y^1_0) T^\top,   TD(y^2_0)T^\top  \right\rangle_F.$$
Recalling that $\langle D(y^1_0), D(y^2_0) \rangle_F = (y^1_0)^\top y_0^2$, 
we use Lemma \ref{lem:matrixapprox} to conclude that with probability at least $1 - 8m^2\exp(-k/2(\varepsilon^2/2-\varepsilon^3/3))$ we have
\begin{equation*}
\left|\langle D(y^1_0), D(y^2_0) \rangle_F - \langle T D(y^1_0) T^\top,   TD(y^2_0)T^\top \rangle_F\right| \le 3\varepsilon \|D(y^1_0)\|_* \|D(y^2_0)\|_F,
\end{equation*}
and hence
\begin{align}\label{eq2}
\left|\sum\limits_{i=1}^m {y^1_0}_ie_i^\top T^\top TD(y^2_0)T^\top Te_i - {y^1_0}^\top y^2_0 \right|&=\left| {y^1_0}^\top y^2_0 - \langle T D(y^1_0) T^\top,   TD(y^2_0)T^\top \rangle_F\right| \notag\\ & \le 3\varepsilon \|y^1_0\|_1 \|y^2_0\|_2 
= 3\varepsilon \frac{\|y^1_0\|_1}{\|y^1_0\|_2} \|y^1_0\|_2 \|y^2_0\|_2.
\end{align}
Now let us bound the second term, $\left|\sum\limits_{i=1}^m {y^1_0}_i\left( (Te_i)^\top\left(  D(D^{-1}(TD(y^2_0)T^\top)) -  TD(y^2_0)T^\top \right)(Te_i) \right)\right|$, of the sum in \eqref{eq1}. According to Proposition \ref{prop:1}, we have that
$$\|D(D^{-1}(TD(y_0^2)T^\top)) - TD({y^2_0})T^\top\|_2 \le 16\varepsilon \frac{\max |{y^2_0}_i|}{k \min |{y^2_0}_i|}\|{y^2_0}\|_1 $$
holds with probability at least $1-4\delta$. Hence for all $i\le m$, we have that 
$$ \left|(Te_i)^\top\left(  D(D^{-1}(TD(y^2_0)T^\top)) -  TD(y^2_0)T^\top \right)(Te_i) \right| \le  16\varepsilon \frac{\max |{y^2_0}_i|}{k \min |{y^2_0}_i|}\|{y^2_0}\|_1 \|Te_i \|_2^2.  $$
Furthermore, by the JLL (Lemma \ref{lem:jll0}), we have that with probability at least $1 - 4m\exp(-k/2(\varepsilon^2/2-\varepsilon^3/3))$,
$\|Te_i \|_2^2 \le 1+\varepsilon$ holds for all $i\le m$. Hence with probability at least $1 - 4\delta - 4m\exp(-k/2(\varepsilon^2/2-\varepsilon^3/3))$\footnote{The probability $1 - 4\delta - 4m\exp(-k/2(\varepsilon^2/2-\varepsilon^3/3))$  is obtained by an union bound between $1-4\delta$ and  $1 - 4m\exp(-k/2(\varepsilon^2/2-\varepsilon^3/3))$, using the fact that $P(E_1\cap E_2)\geq 1 - (2-P(E_1) - P(E_2))$ for any events $E_1$ and $E_2$. } (we remind that $T=\frac{1}{\sqrt{k}}G$), we have that
\begin{equation*}
\left|(Te_i)^\top\left(  D(D^{-1}(TD(y^2_0)T^\top)) -  TD(y^2_0)T^\top \right)(Te_i) \right| \le  16\varepsilon \frac{\max |{y^2_0}_i|}{k \min |{y^2_0}_i|}\|{y^2_0}\|_1(1+\varepsilon)
\end{equation*}
and hence that 
\begin{equation}\label{eq3}
\begin{aligned}
\sum\limits_{i=1}^m\left|{y^1_0}_i(Te_i)^\top\left(  D(D^{-1}(TD(y^2_0)T^\top)) -  TD(y^2_0)T^\top \right)(Te_i) \right| \\\le  16\varepsilon \left(\frac{\|{y^1_0}\|_1\|{y^2_0}\|_1}{k\|{y^1_0}\|_2\|{y^2_0}\|_2}\frac{\max |{y^2_0}_i|}{\min |{y^2_0}_i|}(1+\varepsilon) \right) \|{y^1_0}\|_2\|{y^2_0}\|_2.
\end{aligned}
\end{equation}
By combining \eqref{eq1}, \eqref{eq2}, \eqref{eq3} we have that 
$$|(Sy^1_0)^\top(Sy^2_0) - {y^1_0}^\top y^2_0| \le 16\varepsilon \left(\frac{\|{y^1_0}\|_1}{k\|{y^1_0}\|_2} \left(\frac{\|{y^2_0}\|_1}{\|{y^2_0}\|_2}\frac{\max |{y^2_0}_i|}{\min |{y^2_0}_i|}(1+\varepsilon) +k \right) \right) \|y^1_0\|_2  \|y^2_0\|_2 $$
holds with probability $1 - 4 \delta - (8m^2+4m)\exp(-k/2(\varepsilon^2/2-\varepsilon^3/3))$. 
\end{proof}

Notice that the term $(8m^2+4m)\exp(-k/2(\varepsilon^2/2-\varepsilon^3/3))$ in the probability appearing in Proposition \ref{prop:3} can be made arbitrarily small choosing $k= k_0\frac{\log(m)}{\varepsilon^2}$, for some constant $k_0$. The proof of Proposition \ref{prop:3} requires indeed  $k$ at least equal to $k_0\frac{\log(m)}{3\varepsilon^2-2\varepsilon^3}  \ge k_0\frac{\log(m)}{\varepsilon^2}$, as we have used Lemma \ref{lem:jll0} for $h=m$ in the proof. We now explain how to choose the constant $k_0$ in the $O(\frac{\log(m)}{\varepsilon^2})$ such that  $(8m^2+4m)\exp(-k/2(\varepsilon^2/2-\varepsilon^3/3))$ is small enough.\\

Indeed for any $\delta'\in (0,1)$, since $k=k_0\frac{\log(m)}{\varepsilon^2}$, we have that
$ (8m^2+4m)\exp(-k/2(\varepsilon^2/2-\varepsilon^3/3))\le (8m^2+4m)\exp(-k/12\varepsilon^2)  \le \delta'$ is equivalent to
$$ \frac{(8m^2+4m)}{\exp({\varepsilon^2/12 k})} = \frac{8m^2+4m}{m^{ k_0/12}} \le \delta', $$
which is achieved by taking, for example, $$k_0 \ge  \frac{3+\ln(12) - \ln(\delta')}{1/12} \ge   12\frac{3\ln(m)+\ln(12) - \ln(\delta')}{\ln(m)}       \ge 12\frac{\ln(8 m^2+4m) - \ln(\delta')}{\ln(m)},$$
as $\ln(m)\ge 1$. Hence the condition required in Proposition \ref{prop:1} and \ref{prop:3} is equivalent to
$$m \ge  \frac{2^8}{\mathcal{C}_1\varepsilon^2}\left(3k_0\frac{\ln(m)}{\varepsilon^2}-\ln(\delta)\right),$$
which holds for sufficiently large $m$. 

\begin{proof}[Proof of Proposition~\ref{prop:2}.]
\noindent Let us write $b=(b_0,b_1,\cdots,b_l)$, such that the scalar product $b^\top y^*$ can be decomposed into
\begin{equation}\label{eq:dualval}
b^\top y^*= b_0^\top y^*_0 + \sum\limits_{i=1}^l \langle b_i,  M^*_i\rangle_F.
\end{equation}
To bound $|b^\top y^* - \mathcal{Q}(b)^\top z_Q |= |b^\top y^* - b^\top \mathcal{Q}^\top(\mathcal{Q}(y^*)) |$, we first write
\begin{equation}\label{eq:11}
|b^\top y^* - b^\top \mathcal{Q}^\top(\mathcal{Q}(y^*)) |=\left|b_0^\top y^*_0 - (Sb_0)^\top S y^*_0 + \sum\limits_{i=1}^l \left( \left\langle b_i,M_i^* \right\rangle_F- \left\langle{T^{(i)}}b_i{T^{(i)}}^\top ,  {T^{(i)}} M_i^* {T^{(i)}}^\top \right\rangle_F \right) \right|.
\end{equation}
Using Lemma \ref{lem:matrixapprox}, for all $i\le l$, we can bound the terms
$$\left\langle{T^{(i)}}b_i{T^{(i)}}^\top ,  {T^{(i)}} M_i^* {T^{(i)}}^\top \right\rangle_F  - \left\langle b_i,M_i^* \right\rangle_F. $$
In fact, by an union bound, we obtain that with probability at least $1 - \sum\limits_{i=1}^l 8p_i^2\exp(-q_i/2(\varepsilon^2/2-\varepsilon^3/3))$:
\begin{equation} \label{eq:3a}
\left|\sum\limits_{i=1}^l \left( \left\langle{T^{(i)}}b_i{T^{(i)}}^\top ,  {T^{(i)}} M_i^* {T^{(i)}}^\top \right\rangle_F  - \left\langle b_i,M_i^* \right\rangle_F  \right) \right| \le 3\varepsilon \left(\max \frac{\|M_i^*\|_*}{\|M_i^*\|_F}\right) \sum\limits_{i=1}^l \|M^*_i\|_F \|b_i\|_F.
\end{equation}
Furthermore, using Proposition \ref{prop:3} with $y^1_0=b_0$ and $y^2_0=y_0^* $, we have that with probability at least $1 - 4 \delta - (8m^2+4m)\exp(-k/2(\varepsilon^2/2-\varepsilon^3/3))$ 
\begin{equation}\label{eq:c2}
|b_0^\top S^\top Sy^*_0 - b_0^\top y^*_0| \le \varepsilon \alpha(b_0,y^*_0) \|y^*_0\|_2 \|b_0\|_2.
\end{equation}
Using \eqref{eq:c2} and \eqref{eq:3a} with \eqref{eq:11} we obtain, by an union bound, that Equation \eqref{eq:8} of Proposition~\ref{prop:2} holds with probability at least
$$1-4\delta - (8m^2 +4m)\exp(-k/2(\varepsilon^2/2-\varepsilon^3/3))-\sum\limits_{i=1}^l 8p_i^2 \exp(-q_i/2(\varepsilon^2/2-\varepsilon^3/3)).$$
Now we will bound the term ${A^{(0)}}^\top (S^\top Sy^*_0 - y^*_0)$ from \eqref{eq:error2}. 
For all $i\in \{1,\cdots,n\}$, using Proposition~\ref{prop:3} with $y^1_0=A^{(0)}_i$, $y^2_0=y_0^* $, and taking the $\delta$  of Proposition \ref{prop:3} equal to $\frac{\delta}{n}$, if  $$ m\ge  \frac{2^8}{\mathcal{C}_1\varepsilon^2}\left(3k-\ln(\frac{\delta}{n})\right),$$
we have, by an union bound, that with probability at least  
$$1 - n\left(4\frac{\delta}{n} - (8m^2+4m)\exp(-k/2(\varepsilon^2/2-\varepsilon^3/3))\right),$$
\begin{equation}\label{eq:c1}
|{A^{(0)}}^\top(S^\top Sy^*_0 - y^*_0) | \le \varepsilon \alpha(y^*_0,A^{(0)})\|y^*_0\|_2 \begin{pmatrix}
\|A^{(0)}_1\|_2 \\
\vdots \\
\|A^{(0)}_n\|_2 
\end{pmatrix}
\end{equation}
holds. 


\noindent We are now ready to bound the error $ E=A^\top(\mathcal{Q}^\top(z_Q) - y^*)$. Using  
\eqref{eq:c1}, \eqref{eq:5}, \eqref{eq:error2}, \eqref{eq:6}  we prove by an union bound that with probability at least
$$1-4\delta - (8m^2 +4m)n\exp(-k/2(\varepsilon^2/2-\varepsilon^3/3))-\sum\limits_{i=1}^l 8p_i^2n \exp(-q_i/2(\varepsilon^2/2-\varepsilon^3/3)),$$
\eqref{eq:7} of Proposition~\ref{prop:2} holds. Hence, by an union bound, we prove that the claim of the proposition holds with probability at least
$$1-8\delta - (8m^2 +4m)(n+1)\exp(-k/2(\varepsilon^2/2-\varepsilon^3/3))-\sum\limits_{i=1}^l 8p_i^2(n+1) \exp(-q_i/2(\varepsilon^2/2-\varepsilon^3/3)).$$ 
\end{proof}

\subsection{Bounding the projected optimization problem}
Now we are ready to show the main theorem.

\begin{thm}\label{thm:1}
	Let $\varepsilon,\delta,m$ be such that  $0 <\delta < \frac{1}{8} $,  $0 < \varepsilon \le 1$ and
	$m \ge \frac{2^8}{\mathcal{C}_1\varepsilon^2}(3k + \ln(n) -\ln(\delta)) $. With probability at least $1-8\delta - (8m^2+4m)(n+1)\exp(-k/2(\varepsilon^2/2-\varepsilon^3/3))-\sum\limits_{i=1}^l 8p_i^2(n+1)\exp(-q_i/2(\varepsilon^2/2-\varepsilon^3/3))$, we have
	\begin{align*}
	v(\mathcal{P}^{\mathcal{K}})\left(1- \varepsilon \max \left(\alpha(y^*_0,A^{(0)},b_0), \max\limits_{i=1,\cdots,l} \frac{\|M^*_i\|_*}{\|M^*_i\|_F} \right)\left( \max_{j=1,\cdots,n} \left(\frac{1}{|\cos(\gamma_j)|}\right)\frac{4\|x^*_Q\|_2}{\cos(\theta)\|x^*\|_2} + \frac{3}{\cos(\beta)} \right)\right) \le\\ v(\mathcal{P}^{\mathcal{K}}_\mathcal{Q}) \le  v(\mathcal{P}^{\mathcal{K}}), 
	\end{align*}
	where
	\begin{itemize}
		\item $x^*,(y^*,\lambda^*)$ are optimal solutions of $\mathcal{P}^\mathcal{K}$ and $\mathcal{D}^\mathcal{K}$, respectively,
		\item  $\alpha(y^*_0,A^{(0)},b_0)=\max(\alpha(y^*_0,A^{(0)}),\alpha(y^*_0,b_0))$,
		\item $\beta$ is the angle between $(b,d)$ and $(y^*,\lambda^*)$,  $\gamma_j$ is the angle between $(y^*,\lambda^*)$ and the $j$th column of the matrix $\begin{pmatrix}
		A \\
		B
		\end{pmatrix}$, 	
		\item $\theta$ is the angle between $c$ and $x^*$,
		\item $x^*_Q$ is a feasible solution of $\mathcal{P}^\mathcal{K}_\mathcal{Q}$ such that $c^\top x^*_Q - v(\mathcal{P}^\mathcal{K}_\mathcal{Q})\le \varepsilon'$ for
		some $\varepsilon'$ satisfying 
		\[\varepsilon'\le \varepsilon  \max \left(\alpha(y^*_0,A^{(0)},b_0), \max\limits_{i=1,\cdots,l} \frac{\|M^*_i\|_*}{\|M^*_i\|_F} \right) \max_{j=1,\cdots,n} \left(\frac{1}{|\cos(\gamma_j)|}\right)\frac{\|x^*_Q\|_2}{\cos(\theta)\|x^*\|_2}v(\mathcal{P}^\mathcal{K}).\]
	\end{itemize} 
\end{thm}
Notice that in the case where $v(\mathcal{P}^\mathcal{K})< 0$ we have that $\cos(\beta)$ and $\cos(\theta)$ are negative, implying that Theorem~\ref{thm:1} also holds for the case.
Notice that the probability in the above theorem can be made arbitrarily small by considering
$k=O\left(\frac{\log(m)}{\varepsilon^2}\right)$ and $q_i=O\left(\frac{\log(p_i)}{\varepsilon^2}\right)$.
\begin{proof}
	Let $\varepsilon, \delta, m$ be as in the assumptions of theorem and let $z_Q$ be as in 
	\eqref{eq:zQ}. Since $z_Q$ is a feasible solution of $\mathcal{D}^\eps_Q$, we have by Proposition \ref{prop:2}, that
	$$v(\mathcal{D}^\eps_Q)\ge \transpose{(Q(b))}z_Q + d^\top \lambda^* \ge  \transpose{b}y^*+ d^\top \lambda^* - \varepsilon \left(\alpha(y^*_0,b_0) \|b_0\|_2 \|y^*_0\|_2 + 3 \left(\max\limits_{i=1,\cdots,l} \frac{\|M^*_i\|_*}{\|M^*_i\|_F} \right)\sum\limits_{i=1}^l  \|b_i\|_F\|M_i^*\|_F\right) . $$
	Hence, by Hypothesis \ref{hyp:0}\eqref{hyp:4}, 
	\begin{equation}\label{eq:12}
	v(\mathcal{D}^\eps_Q) \ge v(\mathcal{P}^{\mathcal{K}}) - \varepsilon \left(\alpha(y^*_0,b_0) \|b_0\|_2 \|y^*_0\|_2 + 3 \left(\max\limits_{i=1,\cdots,l} \frac{\|M^*_i\|_*}{\|M^*_i\|_F} \right)\sum\limits_{i=1}^l  \|b_i\|_F\|M_i^*\|_F\right).
	\end{equation}
	We recall that $y^*=(y^*_0,M^*_1,\cdots,M^*_l)$ and $b=(b_0,b_1,\cdots,b_l)$. Hence
	\begin{align*}
	& \|y^*\|_2^2=\|y_0^*\|_2^2+\sum\limits_{i=1}^l\|M^*_i\|_F^2  \\
	& \|b\|_2^2 = \|b_0\|_2^2+\sum\limits_{i=1}^l\|b_i\|_F^2. 
	\end{align*}
	Since
	$$\|y_0^*\|_2\|b_0\|_2+ \sum\limits_{i=1}^l\|M^*_i\|_F\|b_i\|_F \le \|y^*\|_2\|b\|_2,$$
	we deduce by \eqref{eq:12},
	$$v(\mathcal{D}^\eps_Q) \ge v(\mathcal{P}^{\mathcal{K}}) - 3\varepsilon \max \left(\alpha(y^*_0,b_0), \max\limits_{i=1,\cdots,l} \frac{\|M^*_i\|_*}{\|M^*_i\|_F}\right) \|y^*\|_2\|b\|_2.$$
	Since $\|y^*\|_2 \le \|(y^*,\lambda^*)\|_2$ and $\|b\|_2 \le \|(b,d)\|_2$, we have that
	$$v(\mathcal{D}^\eps_Q) \ge v(\mathcal{P}^{\mathcal{K}}) - 3\varepsilon \max \left(\alpha(y^*_0,b_0), \max\limits_{i=1,\cdots,l} \frac{\|M^*_i\|_*}{\|M^*_i\|_F}\right)\|(y^*,\lambda^*)\|_2\|(b,d)\|_2.$$
	Let us denote by $\beta \in [-\pi,\pi]$ the angle between $(b,d)$ and $(y^*,\lambda^*)$.
	That is, $\beta$ satisfies
	\[
	v(\mathcal{D}^{\mathcal{K}}) = b^\top y^* + d^\top \lambda^* = \cos(\beta)\|(y^*,\lambda^*)\|_2\|(b,d)\|_2.
	\]
	By Hypothesis \ref{hyp:0}\eqref{hyp:6} we have 
	$\cos(\beta)\neq 0.$
	In addition, by Hypothesis \ref{hyp:0}\eqref{hyp:4}, 
	we have $v(\mathcal{D}^{\mathcal{K}}) = v(\mathcal{P}^{\mathcal{K}})$,  therefore
	\begin{align*}
	v(\mathcal{D}^\eps_Q)  &\ge v(\mathcal{P}^{\mathcal{K}}) - 3\varepsilon  \max \left(\alpha(y^*_0,b_0), \max\limits_{i=1,\cdots,l} \frac{\|M^*_i\|_*}{\|M^*_i\|_F} \right) \frac{1}{\cos(\beta)} (b^\top y^*+d^\top \lambda^*) \\
	&= v(\mathcal{P}^{\mathcal{K}}) \left( 1 - 3\varepsilon \frac{1}{\cos(\beta)} \max \left(\alpha(y^*_0,b_0), \max\limits_{i=1,\cdots,l} \frac{\|M^*_i\|_*}{\|M^*_i\|_F} \right) \right) .
	\end{align*}
	By weak duality we deduce that 
	\begin{equation}\label{eq:9}
	v(\mathcal{P}^{\mathcal{K}}) \left( 1 - 3\frac{\varepsilon}{\cos(\beta)} \max \left(\alpha(y^*_0,b_0), \max\limits_{i=1,\cdots,l} \frac{\|M^*_i\|_*}{\|M^*_i\|_F} \right) \right) \le v(\mathcal{D}^\eps_Q) \le  v(\mathcal{P}^\eps_Q),
	\end{equation}
	where $\mathcal{P}^\eps_Q$ denotes the dual of $\mathcal{D}^\eps_Q$:
	\begin{equation}\label{eq:lpeps}
	\mathcal{P}^\eps_Q\left\{\begin{array}{llll}
	\min\limits_x & \transpose{(c+E)}x  \\
	& \mathcal{Q}(Ax - b) \in \mathcal{Q}(\mathcal{K}) \\
	& Bx -d \in \mathcal{K}'\\
	& x \in \mathbb{R}^n,
	\end{array}\right.
	\end{equation}
	where $E$ is defined as in \eqref{eq:E}.
	
	\noindent By definition of $\mathcal{Q}$, $\mathcal{P}^\mathcal{K}_\mathcal{Q}$ is a relaxation of $\mathcal{P}^\mathcal{K}$, hence it is feasible. 
	Let $x^*_Q$ be a feasible solution of $\mathcal{P}^\mathcal{K}_\mathcal{Q}$ such that $c^\top x^*_Q - v(\mathcal{P}^\mathcal{K}_\mathcal{Q})\le \varepsilon'$, where 
	\begin{equation}\label{eq:varepsilon}
	\varepsilon'\le \varepsilon  \max \left(\alpha(y^*_0,A^{(0)}), \max\limits_{i=1,\cdots,l} \frac{\|M^*_i\|_*}{\|M^*_i\|_F} \right) \max\left(\frac{1}{|\cos(\gamma_j)|}\right)\frac{\|x^*_Q\|_2}{\cos(\theta)\|x^*\|_2}v(\mathcal{P}^\mathcal{K}).
	\end{equation}
	Such a $x^*_Q$ exists since $\mathcal{P}^\mathcal{K}_\mathcal{Q}$ is feasible and its minimum is bounded by Hypothesis~\ref{hyp:0}\eqref{hyp:1}. Putting such a solution in \eqref{eq:lpeps}, we have that 
	$$ v(\mathcal{P}^\eps_Q) \le c^\top x^*_Q+ \transpose{E} x^*_Q.$$
	From $c^\top x^*_Q - v(\mathcal{P}^\mathcal{K}_\mathcal{Q})\le \varepsilon'$ we deduce that
	\[ v(\mathcal{P}^\eps_Q) \le v(\mathcal{P}^\mathcal{K}_\mathcal{Q})+ \transpose{E} x^*_Q+\varepsilon'.
	\]
	By Proposition \ref{prop:2}, we have that
	$$|E| \le \varepsilon \alpha(y^*_0,A^{(0)}) \begin{pmatrix}
	\|A^{(0)}_1\|_2 \| y^*_0\|_2 \\
	\vdots \\
	\|A^{(0)}_n\|_2 \| y^*_0\|_2
	\end{pmatrix} + 3 \varepsilon \left(\max\limits_{i=1,\cdots,l} \frac{\|M^*_i\|_*}{\|M^*_i\|_F} \right)\sum\limits_{i=1}^l  \|M_i^*\|_F \begin{pmatrix}
	\|A^{(i)}_1\|_F \\
	\vdots \\
	\|A^{(i)}_n\|_F
	\end{pmatrix}.$$
	Hence 
	$$|E| \le 3\varepsilon \max\left(\alpha(y^*_0,A^{(0)}) ,\max\limits_{i=1,\cdots,l} \frac{\|M^*_i\|_*}{\|M^*_i\|_F} \right)  \left(\begin{pmatrix}
	\|A^{(0)}_1\|_2 \| y^*_0\|_2 \\
	\vdots \\
	\|A^{(0)}_n\|_2 \| y^*_0\|_2
	\end{pmatrix} + \sum\limits_{i=1}^l  \|M_i^*\|_F \begin{pmatrix}
	\|A^{(i)}_1\|_F \\
	\vdots \\
	\|A^{(i)}_n\|_F 
	\end{pmatrix}\right).$$
	Since for all $j \in \{1,\cdots,n\}$, we have that 
	$$ \|A^{(0)}_j\|_2 \| y^*_0\|_2 + \sum\limits_{i=1}^l  \|M_i^*\|_F \|A^{(i)}_j\|_F \le \left\| \begin{pmatrix}
	A^{(0)}_j  \\
	\vdots \\
	A^{(l)}_j
	\end{pmatrix} \right\|_2 \left\| \begin{pmatrix}
	y^*_0  \\
	\vdots \\
	M^*_l
	\end{pmatrix} \right\|_2,$$
	we deduce that 
	$$|E| \le  3\varepsilon \max\left(\alpha(y^*_0,A^{(0)}),\max\limits_{i=1,\cdots,l} \frac{\|M^*_i\|_*}{\|M^*_i\|_F} \right)  \begin{pmatrix}
	\|A_1\|_2 \|y^*\|_2 \\
	\vdots \\
	\|A_n\|_2 \|y^*\|_2
	\end{pmatrix},
	$$
	where $A_j=\begin{pmatrix}
	A^{(0)}_j  \\
	\vdots \\
	A^{(l)}_j
	\end{pmatrix}$ is the $j$-th column of $A$.\\
	Consider the columns $C_1,\cdots,C_n$ of the matrix $\begin{pmatrix}
	A \\
	B
	\end{pmatrix}$ for \eqref{eq:lpg2}.
	Since for all $i$, $\|A_i\|_2\le \|C_i\|_2$, we have
	\[
	|E| \le  3\varepsilon \max\left(\alpha(y^*_0,A^{(0)}),\max\limits_{i=1,\cdots,l} \frac{\|M^*_i\|_*}{\|M^*_i\|_F} \right)  \begin{pmatrix}
	\|C_1\|_2 \|(y^*,\lambda^*)\|_2 \\
	\vdots \\
	\|C_n\|_2 \|(y^*,\lambda^*)\|_2
	\end{pmatrix}.
	\]
	Let us consider for all $j\in \{1,\cdots,n\}$ the angles $\gamma_j$ between $(y^*,\lambda^*)$ and $C_j$. Since $C_j^\top (y^*,\lambda^*) = A_j^\top y^* + B_j^\top \lambda^*=c_j \neq 0$, we have that  $\cos(\gamma_j) \neq 0$ for every $j$. Hence we have 
	
	\begin{align}\label{eq:e}
	|E| &\le 3\varepsilon \max_{j=1,\cdots,n} \left(\frac{1}{|\cos(\gamma_j)|}\right)\max\left(\alpha(y^*_0,A^{(0)}) ,\max\limits_{i=1,\cdots,l} \frac{\|M^*_i\|_*}{\|M^*_i\|_F} \right)  \begin{pmatrix}
	|C_1^\top (y^*,\lambda^*)| \\
	\vdots \\
	|C_n^\top (y^*,\lambda^*)|
	\end{pmatrix} \\
	&=3\varepsilon \max_{j=1,\cdots,n} \left(\frac{1}{|\cos(\gamma_j)|}\right)\max\left(\alpha(y^*_0,A^{(0)}) ,\max\limits_{i=1,\cdots,l} \frac{\|M^*_i\|_*}{\|M^*_i\|_F} \right)   \begin{pmatrix}
	|c_1| \\
	\vdots \\
	|c_n|
	\end{pmatrix}.
	\end{align}
	Hence, we have
	\begin{align*}
	|E^\top x^*_Q| \le \|x^*_Q\|_2 \|E\|_2 &\le 3\varepsilon \max\left(\frac{1}{|\cos(\gamma_j)|}\right)\max\left(\alpha(y^*_0,A^{(0)}) ,\max\limits_{i=1,\cdots,l} \frac{\|M^*_i\|_*}{\|M^*_i\|_F} \right)   \|x^*_Q\|_2 \|c\|_2 \\
	&= 3\varepsilon \max\left(\frac{1}{|\cos(\gamma_j)|}\right)\max\left(\alpha(y^*_0,A^{(0)}) ,\max\limits_{i=1,\cdots,l} \frac{\|M^*_i\|_*}{\|M^*_i\|_F} \right)  \frac{\|x^*_Q\|_2}{\|x^*\|_2} \frac{c^\top x^*}{\cos(\theta)},
	\end{align*}
	where $x^*$ is an optimal solution of $v(\mathcal{P}^\mathcal{K})$ and where  $\theta$ is the angle between $c$ and $x^*$. By Hypothesis~\ref{hyp:0}\eqref{hyp:6} we have $\cos(\theta)\neq 0$.\\
	Hence,
	\begin{gather*}
	v(\mathcal{P}^\eps_Q) \le v(\mathcal{P}^\mathcal{K}_\mathcal{Q})+ |\transpose{E} x^*_Q| + \varepsilon'\le \\  v(\mathcal{P}^\mathcal{K}_\mathcal{Q}) +3\varepsilon  \max \left(\alpha(y^*_0,A^{(0)}), \max\limits_{i=1,\cdots,l} \frac{\|M^*_i\|_*}{\|M^*_i\|_F} \right) \max\left(\frac{1}{|\cos(\gamma_j)|}\right)\frac{\|x^*_Q\|_2}{\cos(\theta)\|x^*\|_2}v(\mathcal{P}^\mathcal{K})+\varepsilon'.\end{gather*}
	Now by \eqref{eq:varepsilon} we have that 
	$$  v(\mathcal{P}^\eps_Q) \le  v(\mathcal{P}^\mathcal{K}_\mathcal{Q}) +4\varepsilon  \max \left(\alpha(y^*_0,A^{(0)}), \max\limits_{i=1,\cdots,l} \frac{\|M^*_i\|_*}{\|M^*_i\|_F} \right) \max\left(\frac{1}{|\cos(\gamma_j)|}\right)\frac{\|x^*_Q\|_2}{\cos(\theta)\|x^*\|_2}v(\mathcal{P}^\mathcal{K}).$$
	Combining the inequality above with \eqref{eq:9}, we obtain
	\begin{align*}
	& v(\mathcal{P}^{\mathcal{K}}) \left( 1 - 3\frac{\varepsilon}{\cos(\beta)} \max \left(\alpha(y^*_0,b_0), \max\limits_{i=1,\cdots,l} \frac{\|M^*_i\|_*}{\|M^*_i\|_F} \right) \right) \le v(\mathcal{D}^\eps_Q) \le  v(\mathcal{P}^\eps_Q)  \\
	~~~~~ & \le v(\mathcal{P}^\mathcal{K}_\mathcal{Q}) +4\varepsilon  \max \left(\alpha(y^*_0,A^{(0)}), \max\limits_{i=1,\cdots,l} \frac{\|M^*_i\|_*}{\|M^*_i\|_F} \right) \max\left(\frac{1}{|\cos(\gamma_j)|}\right)\frac{\|x^*_Q\|_2}{\cos(\theta)\|x^*\|_2}v(\mathcal{P}^\mathcal{K}).
	\end{align*}
	Hence
	\begin{align}\label{eq:10}
	v(\mathcal{P}^\mathcal{K}_\mathcal{Q}) \ge & v(\mathcal{P}^{\mathcal{K}})  - 3\frac{\varepsilon}{\cos(\beta)} \max \left(\alpha(y^*_0,b_0), \max\limits_{i=1,\cdots,l} \frac{\|M^*_i\|_*}{\|M^*_i\|_F} \right) v(\mathcal{P}^{\mathcal{K}}) \notag \\
	& - 4\varepsilon  \max \left(\alpha(y^*_0,A^{(0)}), \max\limits_{i=1,\cdots,l} \frac{\|M^*_i\|_*}{\|M^*_i\|_F} \right) \max\left(\frac{1}{|\cos(\gamma_j)|}\right)\frac{\|x^*_Q\|_2}{\cos(\theta)\|x^*\|_2}v(\mathcal{P}^\mathcal{K}).
	\end{align}
	Let $\alpha(y^*_0,A^{(0)},b_0)=\max(\alpha(y^*_0,A^{(0)}),\alpha(y^*_0,b_0))$. We have 
	\begin{gather*}
	\max\left(\alpha(y^*_0,A^{(0)},b_0),\max\limits_{i=1,\cdots,l} \frac{\|M^*_i\|_*}{\|M^*_i\|_F}\right)=\\\max\left(\max \left(\alpha(y^*_0,b_0), \max\limits_{i=1,\cdots,l} \frac{\|M^*_i\|_*}{\|M^*_i\|_F} \right),\max \left(\alpha(y^*_0,A^{(0)}), \max\limits_{i=1,\cdots,l} \frac{\|M^*_i\|_*}{\|M^*_i\|_F} \right)\right),
	\end{gather*}
	hence from \eqref{eq:10}, we obtain that
	$$ v(\mathcal{P}^\mathcal{K}_\mathcal{Q}) \ge v(\mathcal{P}^{\mathcal{K}}) \left( 1 - \varepsilon  \max \left(\alpha(y^*_0,A^{(0)},b_0), \max\limits_{i=1,\cdots,l} \frac{\|M^*_i\|_*}{\|M^*_i\|_F} \right) \left(\frac{3}{\cos(\beta)} 
	+ 4  \max\left(\frac{1}{|\cos(\gamma_j)|}\right)\frac{\|x^*_Q\|_2}{\cos(\theta)\|x^*\|_2} \right)\right),$$
	which finishes the proof. 
\end{proof}

\section{The LP case}\label{sec:6}

In this section we consider the case where we have a pure LP:
\begin{center}
	\begin{minipage}{\textwidth}
		\begin{minipage}{0.48\textwidth}
			\begin{equation}\label{eq:purelp}
			\mathcal{P}\left\{\begin{array}{llll}
			\min\limits_x & \transpose{{c}}x  \\
			& Ax \ge {b} \\
			& Bx \ge d \\
			& x \in \mathbb{R}^n
			\end{array}\right.
			\end{equation}
		\end{minipage}
		\begin{minipage}{0.48\textwidth}
			\begin{equation*}
			\mathcal{P}_S\left\{\begin{array}{llll}
			\min\limits_x & \transpose{{c}}x  \\
			& SAx \ge S{b} \\
			& Bx \ge d \\
			& x \in \mathbb{R}^n
			\end{array}\right.
			\end{equation*}
		\end{minipage}
	\end{minipage}
\end{center}
Under Hypothesis~\ref{hyp:0}, we will show a version of Theorem~\ref{thm:1} with a simplified bound. 
\subsection{A simplified error bound}

The idea is to apply some transformations that preserve the optimal value of  $\mathcal{P}$ and, then, use Theorem~\ref{thm:1} on the transformed problem.
First, let $N$ be an invertible $n\times n$ matrix and let $\mathcal{P}^N$ be the 
problem obtained by replacing $A,B,c$ in $\mathcal{P}$ by $AN, BN, N^\top c $.

We observe that the optimal value of $\mathcal{P}$ and $\mathcal{P}^N$ are the same.
This is because the map $x \mapsto N^{-1}x$ is a bijection between the sets of feasible solutions of $\mathcal{P}$ and $\mathcal{P}^N$. Furthermore, this map preserves the objective function value since $c^\top x = (N^\top c)^\top N^{-1}x$.

With that mind we will now construct a specific matrix $N$. 
We assume that $A$ has full row rank $n$ and without lost of generality, we may assume that 
the first $n$ rows of $A$ are linearly independent. Therefore, $A$ can be divided 
in blocks  as follows
\[
A = \begin{pmatrix}
\hat A \\
\tilde A
\end{pmatrix},
\]
where $\hat A$ is an $n\times n$ invertible matrix and $\tilde A$ is an 
$(m-n) \times n$ matrix. Let 
$$ N= {\hat{A}^{-1}}.$$
Hence
\begin{equation}\label{eq:an}
AN^\eta = \begin{pmatrix}
I_n\\
\tilde AN
\end{pmatrix}.
\end{equation}
This shows that every column of $AN$ can have at most $m-n+1$ nonzero elements.
Now, we recall that if a vector $u \in \mathbb{R}^m$ has at most $k$ elements, then $\|u\|_1 \leq \sqrt{k}\|u\|_2$, which is a consequence of the Cauchy-Schwarz inequality\footnote{Let $v$ be a vector such that $v_i$ is $1,-1$ or $0$ if $u_i$ is positive, negative, or null respectively. Then, $\|u\|_1 = u^\top v \leq \|v\|_2\|u\|_2 \leq \sqrt{k} \|u\|_2$.}.

Let $(AN)_j$ denote the $j$th column of $AN^\eta$. By the preceding discussion we have 
\begin{equation}\label{eq:bound_an}
\frac{\|(AN)_j\|_1}{\|(AN)_j\|_2}\le \sqrt{m-n+1},\qquad \forall j \in \{1, \ldots, n\}.
\end{equation}

Next, we will consider the effect of shifting the constants $b,d$ in $\mathcal{P}^N$ using a vector $v$. Let $\mathcal{P}^{N,v}$ be the problem obtained by replacing 
$b,d$ by $b-ANv, d-ANv$.
Then, assuming that $(N^\top c)^\top v  = 0$, we have the optimal values 
of $\mathcal{P}, \mathcal{P}^N$ and $\mathcal{P}^{N,v}$ all coincide: the 
map $x \mapsto N^{-1}x - v$ is a bijection between the sets of feasible solutions of $\mathcal{P}$ and $\mathcal{P}^{N,v}$. Furthermore, this maps preserves the objective function value since $c^\top x = (N^\top c)^\top (N^{-1}x-v)$.

We now select a specific vector $v$. By adding a small random perturbation to $c$ we can assume w.l.o.g. that all the components of $N^\top c$ are non zeros. In particular, since
$(N^\top c)_1 \neq 0$, the matrix $I_c$ obtained by replacing the first row of $I_n$ by $(N^\top c)^\top$ is still invertible. Let $v$ be the (unique) solution satisfying 
\begin{equation}\label{eq:b'}
I_c v = (0,b_2,\ldots,b_n).
\end{equation}
In view of \eqref{eq:an}, we have that for all $j\le n$, $(ANv)_j=v_j$. Furthermore, by \eqref{eq:b'}, we have that for all $2\le j\le n$, $v_j= b_j$. Hence for all $2\le j\le n$, $(ANv)_j=b_j$.

Hence, $v$ has the property that $(N^\top c)^\top v = 0$ and $(b - ANv)_j = 0$ for 
$j = 2,\ldots, n$. Therefore, $b - ANv$ has at most $m-n+1$ nonzero elements and 
we have the bound
\begin{equation}\label{eq:bound_an2}
\frac{\|b-ANv\|_1}{\|b-ANv\|_2}\le \sqrt{m-n+1}.
\end{equation}
We recall that we also have that
\begin{equation}\label{eq:bound_an1}
\frac{\|(AN)_j\|_1}{\|(AN)_j\|_2}\le \sqrt{m-n+1},\qquad \forall j \in \{1, \ldots, n\},
\end{equation}

We now have all the pieces to prove the following result.

\begin{prop}\label{thm:2}
	Consider problem the $\mathcal{P}$ in \eqref{eq:purelp}, where it is assumed that $A$ has full rank.
	Let $\varepsilon,\delta,m$ be such that  $0 <\delta < \frac{1}{8} $,  $0< \varepsilon \le 1$ and
	${m} \ge \frac{2^8}{\mathcal{C}_1\varepsilon^2}(3k + \ln(n)-\ln(\delta)) $. With probability at least $1-8\delta - (8m^2+4m)(n+1)\exp(-k/2(\varepsilon^2/2-\varepsilon^3/3))$, we have
	\begin{align*}
	v(\mathcal{P})\left(1- 16\varepsilon \sqrt{m-n+1} \left(\frac{\|y^*\|_1}{\|y^*\|_2}\frac{\max|{y^*_i}|}{k\min |{y^*_i}|}(1+\varepsilon) +1 \right) \left( \max\left(\frac{1}{|\cos(\gamma_j)|}\right)\frac{4\|x^*_Q\|_2}{\cos(\theta)\|x^*\|_2} + \frac{3}{\cos(\beta)} \right)\right)  \le\\ 
	v(\mathcal{P}_S) \le  v(\mathcal{P}), 
	\end{align*}
	where 
	\begin{itemize}
		\item $N$ and $v$ are such that \eqref{eq:bound_an2} and \eqref{eq:bound_an1}  hold.
		\item $x^*,(y^*,\lambda^*)$ are optimal solutions of $\mathcal{P}^{N,v}$ and $\mathcal{D}^{N,v}$ (the dual of $\mathcal{P}^{N,v}$), respectively,
		\item $\beta$ is the angle between $(b-ANv,d-ANv)$ and $(y^*,\lambda^*)$,  $\gamma_j$ is the angle between $(y^*,\lambda^*)$ and the $j$th column of the matrix $\begin{pmatrix}
		AN \\
		BN
		\end{pmatrix}$, 	
		\item $\theta$ is the angle between $N^\top c$ and $x^*$,
		\item $x^*_Q$ is an optimal solution of the projected problem  $\mathcal{P}_S^{N,v}$.
	\end{itemize} 
\end{prop}
\begin{proof}
By the preceding discussion, the optimal values of $\mathcal{P}$ and $\mathcal{P}_S$ are equal to the optimal values of the transformed problems $\mathcal{P}^{N,v}$, $\mathcal{P}_S^{N,v}$, respectively. With that in mind, we apply 
Theorem~\ref{thm:1} to  $\mathcal{P}^{N,v}$.	

Notice that in the LP case an optimal solution exists (since the optimal value is finite), hence we can take an optimal solution of the projected problem for $x^*_Q$.
To prove the proposition, all we need to do is to bound all the terms $\frac{\|(AN)_j\|_1}{\|(AN)_j\|_2}$ for $j\in \{1,\cdots,n\}$, and 
$\frac{\|b\|_1}{\|b\|_2}$, that appear in $\alpha(y^*,AN)$ and $\alpha(y^*,b-ANv)$ (see \eqref{alphadef}, \eqref{alphadef2}) in Theorem \ref{thm:1} by $\sqrt{m-n+1}$. These bounds follow from \eqref{eq:bound_an} and \eqref{eq:bound_an2}.
\end{proof}
Next we consider a special case of $\mathcal{P}$ where $d = 0$ and $B = I_n$, and $c>0$.
\begin{center}
	\begin{minipage}{\textwidth}
		\begin{minipage}{0.48\textwidth}
			\begin{equation*}
			\mathcal{P}^\ge\left\{\begin{array}{llll}
			\min\limits_x & \transpose{{c}}x  \\
			& Ax \ge {b} \\
			& x \ge 0
			\end{array}\right.
			\end{equation*}
		\end{minipage}
		\begin{minipage}{0.48\textwidth}
			\begin{equation*}
			\mathcal{P}_S^\ge \left\{\begin{array}{llll}
			\min\limits_x & \transpose{{c}}x  \\
			& SAx \ge S{b} \\
			& x \ge 0
			\end{array}\right.
			\end{equation*}
		\end{minipage}
	\end{minipage}
\end{center}

We will prove, by slightly modifying the proof of Theorem \ref{thm:1}, that we can obtain a bound in this case where the term, $\frac{4\|x^*_Q\|_2}{\cos(\theta)\|x^*\|_2}$, does not appear in the approximation ratio. We have the following theorem:
\begin{thm}\label{thm:3}
	Let $\varepsilon,\delta,m$ be such that  $0 <\delta < \frac{1}{8} $,  $0 < \varepsilon \le 1$ and
	$m \ge \frac{2^8}{\mathcal{C}_1\varepsilon^2}(3k +\ln(n)-\ln(\delta)) $. With probability at least $1-8\delta - (8m^2+4m)(n+1)\exp(-k/2(\varepsilon^2/2-\varepsilon^3/3))$, we have:
	\begin{align*}
	v(\mathcal{P}^\ge)\left(1- 48\max\left(\max\limits_{1\le j\le n}\frac{\|A_j\|_1}{\|A_j\|_2},\frac{\|b\|_1}{\|b\|_2}\right)\left(\frac{\|y^*\|_1}{\|y^*\|_2}\frac{\max|{y^*_i}|}{k\min |{y^*_i}|}(1+\varepsilon) +1 \right) \left(\max\left(\frac{1}{|\cos(\gamma_j)|}\right) + \frac{1}{\cos(\beta)} \right) \right) \le\\ 
	v(\mathcal{P}_S^\ge) \le  v(\mathcal{P}^\ge), 
	\end{align*}
	where
	\begin{itemize}
		\item $x^*,(y^*,\lambda^*)$ are optimal solutions of ${\mathcal{P}^\ge}$ and ${\mathcal{D}^\ge}$ , respectively,
		\item $\beta$ is the angle between $(b,0)$ and $(y^*,\lambda^*)$,  $\gamma_j$ is the angle between $(y^*,\lambda^*)$ and the $j$th column of the matrix $\begin{pmatrix}
		A \\
		I_n
		\end{pmatrix}$ 	
	\end{itemize} 
\end{thm}
\begin{proof}
The proof is basically the same as in Theorem \ref{thm:1}:\\
We define as in Theorem \ref{thm:1}
\begin{equation}\label{eq:lpeps2}
\mathcal{P}^\eps_S\left\{\begin{array}{llll}
\min\limits_x & \transpose{(c+E)}x  \\
& SAx \ge Sb \\
& x \ge 0 
\end{array}\right.
\end{equation}
where  $E= A^\top(S^\top S y^* -y^*)$, since we do not have SDP terms anymore. \\
As in \eqref{eq:9}, we have that
$$  v\left({\mathcal{P}^\ge}\right) \left( 1 - 3\frac{\varepsilon}{\cos(\beta)} \alpha(y^*,b) \right) \le  v(\mathcal{P}^\eps_S).$$
where
$$\alpha(y^*,b) = 16\frac{\|b\|_1}{\|b\|_2} \left(\frac{\|y^*\|_1}{\|y^*\|_2}\frac{\max|{y^*_i}|}{k\min |{y^*_i}|}(1+\varepsilon) +1 \right).$$
Hence 
\begin{equation}\label{eq:14}
 v\left({\mathcal{P}^\ge}\right) \left( 1 - 16\frac{\|b\|_1}{\|b\|_2} \frac{3\varepsilon}{\cos(\beta)} \left(\frac{\|y^*\|_1}{\|y^*\|_2}\frac{\max|{y^*_i}|}{k\min |{y^*_i}|}(1+\varepsilon) +1 \right) \right) \le  v(\mathcal{P}^\eps_S)
\end{equation}
Now, using the fact that $ c>0$, we have, as in \eqref{eq:e} that
$$	|E| \le \varepsilon \max\left(\frac{1}{|\cos(\gamma_j)|}\right)\left(3\alpha(y^*,A)\right)   \begin{pmatrix}
c_1 \\
\vdots \\
c_n
\end{pmatrix},$$
hence since $x^*_Q \ge 0$, we obtain that
$$|E^\top x^*| \le 
\varepsilon \max\left(\frac{1}{|\cos(\gamma_j)|}\right)\left(3\alpha(y^*,A) \right)   c^\top x^*_Q ,$$
where $x^*_Q$ is an optimal solution of $\mathcal{P}^\ge_S$.
Hence
\begin{equation}\label{eq:17}
|E^\top x^*| \le  16  \varepsilon \max\limits_{1\le j \le n}\frac{\|A_j\|_1}{\|A_j\|_2} \left(\frac{\|y^*\|_1}{\|y^*\|_2}\frac{\max|{y^*_i}|}{k\min |{y^*_i}|}(1+\varepsilon) +1 \right)3\max\left(\frac{1}{|\cos(\gamma_j)|}\right)  v(\mathcal{P}^\ge_S).
\end{equation}
Since 
$$ v(\mathcal{P}^\eps_S) \le v(\mathcal{P}^\ge_S)+ |\transpose{E} x^*_Q|,$$
and that $v(\mathcal{P}^\ge_S) \le v(\mathcal{P}^\ge)$, we obtain the theorem by combining \eqref{eq:14} and \eqref{eq:17}.
\end{proof}

\subsection{Interpretation of the error bound}
Theorem \ref{thm:3} shows that the approximation ratio $R=\frac{v(\mathcal{P}^\ge)-v(\mathcal{P}^\ge_S)}{v(\mathcal{P}^\ge)}$  is given by
$$ R= O\left( \varepsilon \sigma \left(\frac{\sqrt{n}}{k}+1\right) \max\left(\frac{1}{|\cos(\gamma_j)|}\right)\right),$$
where $\sigma \ge \max\left(\max\limits_{1\le j\le n}\frac{\|A_j\|_1}{\|A_j\|_2},\frac{\|b\|_1}{\|b\|_2}\right) $, as $\frac{\|y^*\|_1}{\|y^*\|_2} \le \sqrt{n}$ since $y^*$can be chosen to have at most $n$ non-zero components, and as the other terms in the error bound can be interpreted as constants that do not depend on the dimension $m,n$ of the problem. 
Notice that $\sigma$ is a bound on the sparsity of the columns of $A$ and the vector $b$.
Furthermore, the probability bound $1-8\delta - (8m^2+4m)(n+1)\exp(-k/2(\varepsilon^2/2-\varepsilon^3/3))$ can be made as close to $1$ as we want by choosing $k=O(\frac{\log(m)}{\varepsilon^2})$. Let $\gamma^*\in \argmin |\cos(\gamma_j)|$, we obtain that
$$ R=\frac{v(\mathcal{P}^\ge)-v(\mathcal{P}^\ge_S)}{v(\mathcal{P}^\ge)} =O\left(\varepsilon \frac{\sigma}{\cos(\gamma^*)} \left( \varepsilon^2\frac{\sqrt{n}}{\log(m)}+1\right) \right). $$
Let us take $$\varepsilon = O\left(\sqrt{\frac{\log(m)}{n^{1/2-\alpha}}}\right),$$
for $0\le \alpha \le \frac{1}{2}$. Then $\varepsilon^2\frac{\sqrt{n}}{\log(m)}=O(n^{\alpha})$, which implies that 
$$ R=O\left(n^\alpha\frac{\sqrt{\sigma^2\log(m)}}{{n}^{1/4-\alpha/2}\cos(\gamma^*)}\right)=O\left(n^{3\alpha/2-1/4}\frac{\sqrt{\sigma^2\log(m)}}{\cos(\gamma^*)}\right). $$

This suggests that taking $\varepsilon = O\left(\sqrt{\frac{\log(m)}{n^{1/2-\alpha}}}\right)$ allows us to obtain a ratio $R$ if 
$$R\cos(\gamma^*)\ge \mathcal{C} {n^{3\alpha/2-1/4}\sqrt{\log(m)\sigma^2}},$$
for some constant $\mathcal{C}$. \\

For such $\varepsilon$, we have $k=O(\frac{\log(m)}{\varepsilon^2})=O(n^{1/2-\alpha})$. Furthermore, the condition 
$m \ge \frac{2^8}{\mathcal{C}_1\varepsilon^2}(3k +\log(n)-\ln(\delta)) $
implies that 
$$m \ge \mathcal{C}'\left(\frac{n^{1-2\alpha}}{\log(m)} + \log(n) -\log(\delta)\right),$$
for some constant $\mathcal{C}'$, which holds, for all $0\le \alpha \le 1/2$ as long as $m,n$ are large enough.\\

The above discussion also suggests that the error bound decreases  when the columns of $A$ and the vector $b$ are sparse.

\section{Numerical results}\label{sec:5}
In this section, we present some preliminary numerical experiments where we generate random instances and we solve both the original formulation and the smaller reduced version. 
A difficulty in performing these experiments is that, although for $0<\varepsilon<1$ fixed, the bound $m\ge  \frac{2^8}{\mathcal{C}_1\varepsilon^2}\left(3k + \ln(m)-\ln(\delta)\right)$ is always satisfied for sufficiently large $m$, 
such $m$ is too large for the computer we are using to solve the original LP. Nevertheless we still perform some experiments on the pure LP (as in Section~\ref{sec:6})  with $m$ up to $20000$. 
All results have been obtained using Gurobi called 
through Julia \cite{julia} and the JuMP \cite{jump} interface. The specs of the machine are as follows: Intel Core i5 at 3.8GHz with 8 GB of DDR4 RAM.

\subsection{Random instances}
The random instances considered here are all feasible: $c$ is the all ones vector, $A$ is a random matrix build either from the uniform or the normal distribution, $b=Ax_0 - \eta$ where both $x_0$ and $\eta$ are random positive vectors and  $\{x \mid Bx -d \in \mathcal{K}'\}$ is just the non-negative orthant. The results are summarized in the tables below, ``$m$'' denotes the number of constraints, ``$n$'' the number of variables, ``$k$'' the number of constraints of the projected problem (computed for $\varepsilon=0.2$), ``$d$'' the density of matrix $A$, ``law'' is the probability law used to generate the coefficient of $A$. Here $U(a,b)$ denotes the uniform law in the interval $[a,b]$ and $N(a,b)$ denotes the normal law of mean $a$ and standard deviation $b$. Each line of the table is the average over 10 instances generated with the same $m$, $n$, $d$, law. Furthermore, ``meantorg'' is the average time to solve the original LP, ``stdtorg'' is the corresponding standard deviation, ``meantproj'' and ``stdtproj'' are respectively the average time and the standard deviation to solve the projected problem, ``meanratio'' is the average error ratio $\frac{v(P)-v(P_S)}{v(P)}$ and ``stdratio'' is the corresponding standard deviation.

\begin{table}[tb]
	\centering
	\tiny
	\begin{tabular}{llllllllll}
		\hline
		m & n & k & law & meantorg{[}s{]} & stdtorg{[}s{]} & meantproj{[}s{]} & stdtproj{[}s{]} & meanratio & stdratio \\
		\hline
		5000  & 1200  & 321 & U(-1,1) & 1.19E+01 & 3.17E-01 & 6.46E-01 & 8.44E-03 & 1.00E+00 & 0.00E+00 \\
		5000  & 1200  & 321 & U(-1,2) & 1.06E+01 & 3.33E-02 & 5.92E-01 & 2.66E-05 & \textbf{1.75E-01} & 1.31E-04 \\
		5000  & 1200  & 321 & U(0,1)  & 1.05E+01 & 9.52E-02 & 5.94E-01 & 3.32E-05 & \textbf{1.12E-01} & 4.62E-05 \\
		5000  & 1200  & 321 & U(0,2)  & 1.06E+01 & 5.19E-02 & 5.94E-01 & 1.39E-05 & \textbf{1.05E-01} & 6.72E-05 \\
		5000  & 1200  & 321 & U(1,2)  & 1.05E+01 & 8.85E-02 & 5.92E-01 & 1.01E-05 &\textbf{ 8.79E-02} & 1.20E-05 \\
		5000  & 1200  & 321 & N(0,1)  & 1.06E+01 & 1.27E-01 & 5.91E-01 & 2.82E-05 & 1.00E+00 & 0.00E+00 \\
		5000  & 1200  & 321 & N(0,2)  & 1.08E+01 & 1.29E-01 & 5.95E-01 & 5.15E-05 & 1.00E+00 & 0.00E+00 \\
		5000  & 1200  & 321 & N(1,2)  & 1.09E+01 & 8.92E-02 & 5.97E-01 & 4.97E-05 & \textbf{1.81E-01} & 4.36E-05 \\ \cline{2-10}
		5000  & 1500  & 331 & U(-1,1) & 1.22E+01 & 2.65E-01 & 7.93E-01 & 3.78E-05 & 1.00E+00 & 0.00E+00 \\
		5000  & 1500  & 331 & U(-1,2) & 1.19E+01 & 1.06E-01 & 7.91E-01 & 6.57E-05 & \textbf{1.73E-01} & 7.14E-05 \\
		5000  & 1500  & 331 & U(0,1)  & 1.20E+01 & 1.02E-01 & 8.00E-01 & 8.39E-04 & \textbf{1.11E-01} & 3.95E-05 \\
		5000  & 1500  & 331 & U(0,2)  & 1.18E+01 & 1.46E-01 & 7.94E-01 & 1.63E-04 & \textbf{1.00E-01} & 1.62E-05 \\
		5000  & 1500  & 331 & U(1,2)  & 1.19E+01 & 1.22E-01 & 7.96E-01 & 1.70E-04 & \textbf{8.73E-02} & 4.46E-06 \\
		5000  & 1500  & 331 & N(0,1)  & 1.20E+01 & 1.06E-01 & 7.95E-01 & 1.56E-04 & 1.00E+00 & 0.00E+00 \\
		5000  & 1500  & 331 & N(0,2)  & 1.17E+01 & 8.17E-02 & 7.91E-01 & 3.86E-05 & 1.00E+00 & 4.97E-08 \\
		5000  & 1500  & 331 & N(1,2)  & 1.20E+01 & 7.10E-02 & 7.96E-01 & 3.88E-05 & \textbf{1.83E-01} & 2.96E-05 \\\cline{2-10}
		5000  & 4000  & 375 & U(-1,1) & 1.88E+01 & 1.58E-01 & 2.40E+00 & 9.18E-03 & 1.00E+00 & 0.00E+00 \\
		5000  & 4000  & 375 & U(-1,2) & 2.25E+01 & 9.32E-01 & 2.37E+00 & 2.26E-04 & \textbf{1.66E-01} & 2.66E-05 \\
		5000  & 4000  & 375 & U(0,1)  & 2.22E+01 & 7.15E-01 & 2.37E+00 & 1.50E-04 & \textbf{1.02E-01} & 1.73E-05 \\
		5000  & 4000  & 375 & U(0,2)  & 2.29E+01 & 2.81E+00 & 2.38E+00 & 2.73E-04 & \textbf{1.01E-01} & 1.26E-05 \\
		5000  & 4000  & 375 & U(1,2)  & 2.22E+01 & 1.36E+00 & 2.40E+00 & 2.96E-03 & \textbf{9.02E-02} & 7.77E-06 \\
		5000  & 4000  & 375 & N(0,1)  & 1.81E+01 & 1.13E-01 & 2.38E+00 & 1.77E-03 & 1.00E+00 & 0.00E+00 \\
		5000  & 4000  & 375 & N(0,2)  & 1.86E+01 & 6.02E-01 & 2.39E+00 & 4.61E-04 & \textbf{9.94E-01} & 2.67E-05 \\
		5000  & 4000  & 375 & N(1,2)  & 2.11E+01 & 6.64E-01 & 2.39E+00 & 2.37E-03 & \textbf{1.86E-01} & 2.88E-05 \\ \hline
		10000 & 1200  & 321 & U(-1,1) & 6.17E+01 & 8.51E+00 & 1.58E+00 & 1.38E-04 & 1.00E+00 & 0.00E+00 \\
		10000 & 1200  & 321 & U(-1,2) & 6.12E+01 & 7.31E+00 & 1.58E+00 & 6.16E-04 & \textbf{1.39E-01} & 8.37E-05 \\
		10000 & 1200  & 321 & U(0,1)  & 5.96E+01 & 5.59E+00 & 1.58E+00 & 4.50E-05 & \textbf{9.20E-02} & 7.34E-06 \\
		10000 & 1200  & 321 & U(0,2)  & 5.99E+01 & 1.62E+00 & 1.58E+00 & 1.03E-05 & \textbf{7.99E-02} & 2.00E-05 \\
		10000 & 1200  & 321 & U(1,2)  & 6.09E+01 & 6.91E+00 & 1.58E+00 & 2.72E-04 &\textbf{ 6.83E-02} & 5.96E-06 \\
		10000 & 1200  & 321 & N(0,1)  & 6.24E+01 & 2.33E+01 & 1.58E+00 & 2.11E-04 & 1.00E+00 & 0.00E+00 \\
		10000 & 1200  & 321 & N(0,2)  & 5.96E+01 & 4.04E+00 & 1.58E+00 & 1.80E-04 & 1.00E+00 & 0.00E+00 \\
		10000 & 1200  & 321 & N(1,2)  & 5.89E+01 & 5.58E+00 & 1.58E+00 & 3.48E-06 & \textbf{1.37E-01} & 4.95E-05 \\\cline{2-10}
		10000 & 1500  & 331 & U(-1,1) & 6.45E+01 & 9.89E+00 & 2.03E+00 & 2.31E-05 & 1.00E+00 & 0.00E+00 \\
		10000 & 1500  & 331 & U(-1,2) & 6.48E+01 & 7.24E+00 & 2.03E+00 & 1.47E-05 & \textbf{1.37E-01} & 1.09E-05 \\
		10000 & 1500  & 331 & U(0,1)  & 6.16E+01 & 4.12E+00 & 2.03E+00 & 2.14E-05 & \textbf{9.12E-02} & 2.02E-05 \\
		10000 & 1500  & 331 & U(0,2)  & 6.22E+01 & 4.49E+00 & 2.05E+00 & 3.94E-04 & \textbf{7.90E-02} & 8.12E-06 \\
		10000 & 1500  & 331 & U(1,2)  & 6.36E+01 & 1.18E+01 & 2.04E+00 & 1.86E-04 & \textbf{6.75E-02} & 6.20E-06 \\
		10000 & 1500  & 331 & N(0,1)  & 6.51E+01 & 3.83E+00 & 2.04E+00 & 1.86E-05 & 1.00E+00 & 0.00E+00 \\
		10000 & 1500  & 331 & N(0,2)  & 6.37E+01 & 7.80E+00 & 2.03E+00 & 5.35E-06 & 1.00E+00 & 0.00E+00 \\
		10000 & 1500  & 331 & N(1,2)  & 6.25E+01 & 3.05E+00 & 2.03E+00 & 2.06E-05 & \textbf{1.38E-01} & 1.31E-05 \\\cline{2-10}
		10000 & 4000  & 375 & U(-1,1) & 9.22E+01 & 4.82E+00 & 5.84E+00 & 1.97E-05 & 1.00E+00 & 0.00E+00 \\
		10000 & 4000  & 375 & U(-1,2) & 1.00E+02 & 9.56E+00 & 5.86E+00 & 1.73E-04 & \textbf{1.35E-01} & 2.08E-05 \\
		10000 & 4000  & 375 & U(0,1)  & 9.64E+01 & 1.14E+01 & 5.87E+00 & 2.10E-03 & \textbf{8.47E-02} & 5.29E-06 \\
		10000 & 4000  & 375 & U(0,2)  & 1.02E+02 & 1.64E+01 & 5.88E+00 & 1.86E-03 & \textbf{8.05E-02} & 3.51E-06 \\
		10000 & 4000  & 375 & U(1,2)  & 1.05E+02 & 2.52E+01 & 5.86E+00 & 1.93E-04 & \textbf{7.05E-02} & 5.04E-06 \\
		10000 & 4000  & 375 & N(0,1)  & 9.10E+01 & 1.35E+01 & 5.87E+00 & 9.09E-04 & 1.00E+00 & 0.00E+00 \\
		10000 & 4000  & 375 & N(0,2)  & 9.22E+01 & 7.08E+00 & 5.85E+00 & 3.24E-05 & 1.00E+00 & 6.88E-07 \\
		10000 & 4000  & 375 & N(1,2)  & 1.01E+02 & 1.62E+01 & 5.86E+00 & 1.92E-04 & \textbf{1.48E-01} & 1.64E-05 \\\cline{2-10}
		10000 & 7000  & 400 & U(-1,1) & 1.19E+02 & 1.37E+01 & 1.14E+01 & 2.41E-04 & 1.00E+00 & 0.00E+00 \\
		10000 & 7000  & 400 & U(-1,2) & 1.41E+02 & 9.06E+01 & 1.15E+01 & 6.54E-03 & \textbf{1.33E-01} & 1.47E-05 \\
		10000 & 7000  & 400 & U(0,1)  & 1.41E+02 & 4.88E+01 & 1.15E+01 & 7.41E-04 & \textbf{8.14E-02} & 4.63E-06 \\
		10000 & 7000  & 400 & U(0,2)  & 1.59E+02 & 3.16E+02 & 1.15E+01 & 1.53E-04 & \textbf{8.11E-02} & 3.38E-06 \\
		10000 & 7000  & 400 & U(1,2)  & 1.60E+02 & 7.92E+02 & 1.15E+01 & 4.30E-04 & \textbf{7.09E-02} & 4.31E-06 \\
		10000 & 7000  & 400 & N(0,1)  & 1.18E+02 & 1.42E+01 & 1.15E+01 & 4.12E-04 & 1.00E+00 & 0.00E+00 \\
		10000 & 7000  & 400 & N(0,2)  & 1.19E+02 & 3.51E+01 & 1.15E+01 & 1.38E-04 & 9.95E-01 & 1.29E-05 \\
		10000 & 7000  & 400 & N(1,2)  & 1.32E+02 & 4.61E+01 & 1.15E+01 & 4.98E-03 & \textbf{1.49E-01} & 1.40E-05 \\\cline{2-10}
		10000 & 9000  & 411 & U(-1,1) & 1.35E+02 & 3.32E+01 & 1.52E+01 & 6.82E-04 & 1.00E+00 & 0.00E+00 \\
		10000 & 9000  & 411 & U(-1,2) & 1.53E+02 & 3.20E+01 & 1.52E+01 & 1.18E-03 & \textbf{1.31E-01} & 1.50E-05 \\
		10000 & 9000  & 411 & U(0,1)  & 1.65E+02 & 3.38E+02 & 1.52E+01 & 1.21E-03 & \textbf{7.80E-02} & 4.61E-06 \\
		10000 & 9000  & 411 & U(0,2)  & 1.75E+02 & 3.60E+02 & 1.52E+01 & 1.32E-03 & \textbf{7.86E-02} & 3.17E-06 \\
		10000 & 9000  & 411 & U(1,2)  & 1.84E+02 & 5.38E+02 & 1.53E+01 & 1.60E-02 & \textbf{6.91E-02} & 4.00E-06 \\
		10000 & 9000  & 411 & N(0,1)  & 1.32E+02 & 3.25E+01 & 1.51E+01 & 4.11E-02 & 1.00E+00 & 5.48E-08 \\
		10000 & 9000  & 411 & N(0,2)  & 1.31E+02 & 1.68E+01 & 1.51E+01 & 1.11E-02 & 9.94E-01 & 5.36E-06 \\
		10000 & 9000  & 411 & N(1,2)  & 1.50E+02 & 4.74E+01 & 1.52E+01 & 6.04E-03 & \textbf{1.44E-01} & 1.90E-05 \\\hline
		20000 & 10000 & 416 & U(0,1)  & 8.32E+02 & 3.12E+03 & 4.36E+01 & 3.72E-01 & \textbf{6.47E-02} & 5.04E-06 \\
		20000 & 10000 & 416 & N(0,1)  & 7.21E+02 & 2.48E+02 & 4.33E+01 & 1.11E-01 & 1.00E+00 & 0.00E+00 \\
		\hline
	\end{tabular}
	\caption{Numerical results for $d=0.1$. The values where $meanratio < 0.2$ are written in boldface}
	\label{table:2}
\end{table}

\clearpage

First we notice, from Table \ref{table:2}, that the time to solve the projected LP is always shorter than the time to solve the original LP, and furthermore as the number of constraints, $m$, of the original LP increases the gap between the two times increases drastically. Concerning the average error ratio between the values of the two problems, the only striking observation we can make from Table \ref{table:2} is that the projected problem approximates the original well when the probability law, used to generate the coefficients of the constraint matrix $A$, does not have $0$ in expectation. Such ratio are written in boldface in Table \ref{table:2}. They correspond to the case when $meanratio <0.2$.\\
In Table \ref{table:3} we see how the standard deviation of the probability law affects the ratio. We see that in both cases (the expectation of the probability law is equal to $0$ or different from $0$) an increase of the standard deviation seems to decrease a little the error ratio.
\begin{table}[tb]
	\centering
	\tiny
	\begin{tabular}{llllllllll}
		\hline
		m     & n    & k    & law       & meantorg[s] & stdtorg[s]  & meantproj[s] & stdtproj[s] & meanratio & stdratio \\ \hline
		5000  & 2000 & 344 & U(0,1)    & 9.31E+01 & 1.53E+04 & 1.02E+01  & 3.50E-02 & 1.87E-02  & 4.19E-07 \\
		5000  & 2000 & 344  & U(0,10)   & 5.51E+01 & 5.39E+00 & 1.02E+01  & 3.53E-03 & 1.76E-02  & 5.65E-07 \\
		5000  & 2000 & 344  & U(0,50)   & 1.11E+02 & 1.57E+04 & 1.01E+01  & 9.34E-03 & 1.82E-02  & 3.30E-07 \\
		5000  & 2000 & 344 & U(0,100)  & 1.21E+02 & 1.92E+04 & 1.02E+01  & 1.07E-04 & 1.79E-02  & 1.93E-07 \\
		5000  & 2000 & 344  & U(0,500)  & 8.15E+01 & 8.71E+00 & 1.02E+01  & 2.22E-04 & 1.80E-02  & 5.42E-07 \\
		5000  & 2000 & 344 & U(0,1000) & 8.97E+01 & 2.32E+01 & 1.03E+01  & 1.00E-02 & 1.77E-02  & 4.90E-07 \\ \cline{2-10}
		5000  & 2000 & 344  & N(0,1)    & 8.59E+01 & 1.19E+04 & 1.02E+01  & 1.24E-03 & 9.94E-01  & 4.88E-05 \\
		5000  & 2000 & 344  & N(0,10)   & 5.25E+01 & 5.42E+00 & 1.02E+01  & 7.57E-05 & 9.71E-01  & 5.10E-05 \\
		5000  & 2000 & 344  & N(0,50)   & 5.38E+01 & 1.65E+00 & 1.03E+01  & 1.02E-02 & 9.76E-01  & 2.94E-05 \\
		5000  & 2000 & 344  & N(0,100)  & 5.48E+01 & 1.68E+00 & 1.02E+01  & 1.85E-04 & 9.73E-01  & 2.03E-05 \\
		5000  & 2000 & 344  & N(0,500)  & 1.04E+02 & 1.91E+04 & 1.02E+01  & 8.15E-04 & 9.69E-01  & 3.49E-05 \\ \hline
		5000  & 2000 & 344 & U(0,1)    & 8.58E+02 & 8.21E+03 & 9.30E+01  & 1.85E-01 & 1.42E-02  & 2.21E-07 \\
		10000 & 7000 & 400  & U(0,10)   & 8.62E+02 & 7.36E+03 & 9.29E+01  & 5.14E-02 & 1.42E-02  & 2.95E-07 \\
		10000 & 7000 & 400  & U(0,100)  & 1.06E+03 & 2.60E+03 & 9.31E+01  & 1.02E-01 & 1.42E-02  & 2.43E-07 \\
		10000 & 7000 & 400  & U(0,1000) & 1.23E+03 & 4.87E+03 & 9.33E+01  & 6.28E-02 & 1.43E-02  & 3.63E-07 \\\cline{2-10}
		10000 & 7000 & 400  & N(0,1)    & 1.79E+03 & 1.03E+07 & 9.33E+01  & 2.05E-01 & 9.91E-01  & 3.99E-05 \\
		10000 & 7000 & 400  & N(0,10)   & 7.15E+02 & 2.94E+02 & 9.33E+01  & 9.72E-02 & 9.85E-01  & 7.17E-06 \\
		10000 & 7000 & 400  & N(0,100)  & 7.97E+02 & 5.36E+02 & 9.33E+01  & 1.39E-01 & 9.83E-01  & 2.22E-05 \\
		10000 & 7000 & 400  & N(0,1000) & 8.52E+02 & 1.49E+03 & 9.32E+01  & 6.74E-02 & 9.79E-01  & 1.82E-05 \\
		\hline
	\end{tabular}
	\caption{Numerical results when the standard deviation of the law changes ($d=1$)}
	\label{table:3}
\end{table}

%
%
%

\begin{figure}[tb]
	\begin{minipage}[b]{0.5\columnwidth}
		\centering
		\begin{tikzpicture}
		\begin{axis}[
		xlabel=$n$ , 
		ylabel=meanratio ,
		legend style={font=\tiny},
		width=0.95\linewidth,
		y tick label style={
			/pgf/number format/.cd,
			fixed,
			fixed zerofill,
			precision=3,
			/tikz/.cd
		},
		x tick label style={
			/pgf/number format/.cd,
			fixed,
			fixed zerofill,
			precision=1,
			/tikz/.cd
		}
		]
		\addplot[dashed, mark=*,mark size=2pt,color=blue] %
		table[x=n,y=r,col sep=comma]{5U.csv};
		\addlegendentry{\tiny $m=5000$ Uniform law};
		
		\addplot[dashed, mark=*,mark size=2pt,color=red] %
		table[x=n,y=r,col sep=comma]{7U.csv};
		\addlegendentry{\tiny $m=7000$ Uniform law};
		
		\end{axis}
		\end{tikzpicture}
		\caption{Average approximation ratio  (uniform case)}
		\label{approxratio_uniform}
	\end{minipage}
	\begin{minipage}[b]{0.5\columnwidth}
		\begin{tikzpicture}
		\begin{axis}[
		xlabel=$n$ , 
		ylabel=meanratio ,
		width=0.95\linewidth,
	y tick label style={
		/pgf/number format/.cd,
		fixed,
		fixed zerofill,
		precision=3,
		/tikz/.cd
	},
	x tick label style={
		/pgf/number format/.cd,
		fixed,
		fixed zerofill,
		precision=1,
		/tikz/.cd
	}
		]
		\addplot[dashed, mark=*,mark size=2pt,color=blue] %
		table[x=n,y=r,col sep=comma]{5N.csv};
		\addlegendentry{\tiny $m=5000$ Normal law};
		
		\addplot[dashed, mark=*,mark size=2pt,color=red] %
		table[x=n,y=r,col sep=comma]{7N.csv};
		\addlegendentry{\tiny $m=7000$ Normal law};
		
		\end{axis}
		\end{tikzpicture}
		\caption{Average approximation ratio (normal case)}
		\label{approxratio_norm}
	\end{minipage}
\end{figure}

Next we show, for $m$ fixed, the influence of $n$ on the value of the approximation ratio.
Figure~\ref{approxratio_uniform} shows how the value of the approximation ratio changes, in the case of an uniform law, in function of $n$ for $m$ fixed.
Although the approximation ratio is catastrophic, in the case of a normal law, 
Figure~\ref{approxratio_norm} shows how the value of the approximation ratio changes in function of $n$ for $m$ fixed.
As suggested by Theorem \ref{thm:2}, we notice that the error bound gets better when $n$ increases for $m$ fixed.
Finally we plot the ratio $\frac{meantproj}{meantorg}$ in function of $n$ for $m$ fixed in Figure~\ref{comptime}. We see that, both in the normal case and in the uniform case the ratio tends to get smaller as $n$ and $m$ increase.

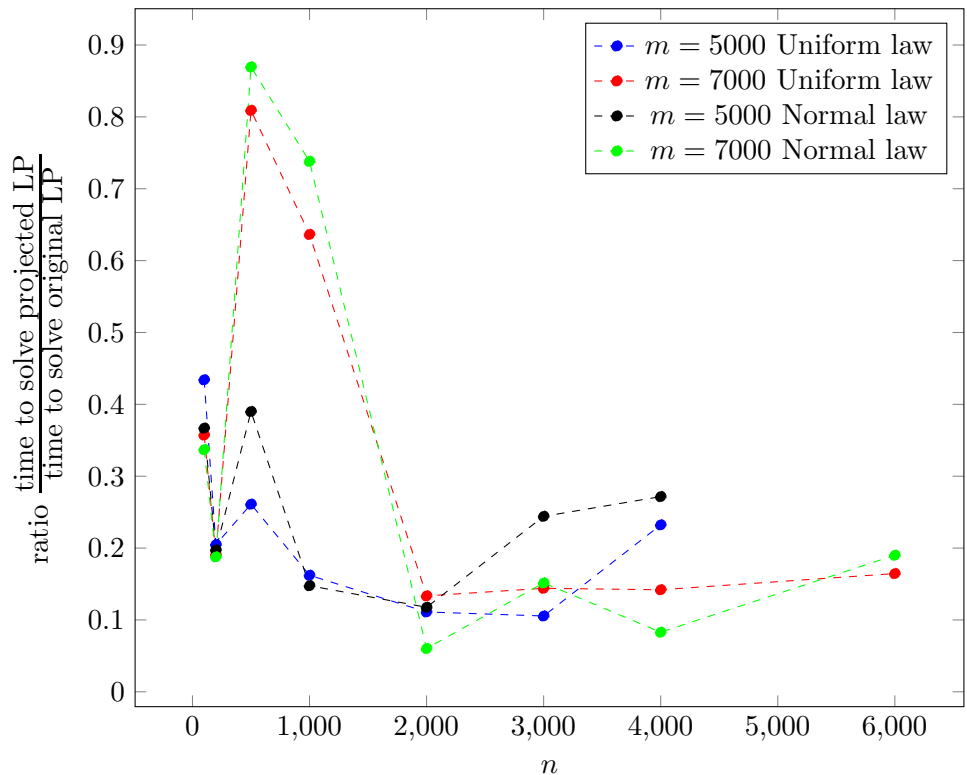
\begin{figure}[tb]
	\centering
	\begin{tikzpicture}
	\begin{axis}[
	xlabel=$n$ , 
	ylabel=ratio $\frac{\mbox{time to solve projected LP}}{\mbox{time to solve original LP}}$ ,
	width=0.75\linewidth,
	]
	\addplot[dashed, mark=*,mark size=2pt,color=blue] %
	table[x=n,y=tratio,col sep=comma]{t5U.csv};
	\addlegendentry{$m=5000$ Uniform law};
	
	\addplot[dashed, mark=*,mark size=2pt,color=red] %
	table[x=n,y=tratio,col sep=comma]{t7U.csv};
	\addlegendentry{$m=7000$ Uniform law};
	
	\addplot[dashed, mark=*,mark size=2pt,color=black] %
	table[x=n,y=tratio,col sep=comma]{t5N.csv};
	\addlegendentry{$m=5000$ Normal law};
	
	\addplot[dashed, mark=*,mark size=2pt,color=green] %
	table[x=n,y=tratio,col sep=comma]{t7N.csv};
	\addlegendentry{$m=7000$ Normal law};
	
	\end{axis}
	\end{tikzpicture}
	\caption{The ratio of computation times between meantorg and meantproj}
	\label{comptime}
\end{figure}

\subsection{Real instances}
We now discuss numerical results on $3$ real instances taken from the \href{http://plato.asu.edu/ftp/lptestset}{Hans Mittelmann collection} (http://plato.asu.edu/ftp/lptestset/). All these instances are LP relaxations from Integer Linear Problem (ILP) taken from \cite{miplib}.

The first instance we consider, \textit{a2864-99blp}, is the LP-relaxation of a clique problem. The problem has $n=200787$ variables all belonging to $[0,1]$ and $m=22117$ constraints. It has a nonzero density equal to 
$4.52\times 10^{-3}$ and is originally solved in $1056.23$ seconds. Using $\varepsilon=0.2$ the projected problem has $k=452$ constraints and was solved in $103.82$ seconds for an approximation ratio equal to $8.87\times 10^{-2}$.
 
 The second instance we consider, \textit{rmine15} comes from open pit mining over a cube. The problem has $n=42438$ variables all belonging to $[0,1]$ and $m=358395$ constraints. It has a nonzero density equal to $5.78407\times 10^{-5}$ and is originally solved in $623.20$ seconds. Using $\varepsilon=0.2$ the projected problem has $k=577$ constraints and was solved in $18.20$ seconds for an approximation ratio equal to $7.72 \times 10^{-1}$. 
 
 The third instance \textit{scpm1} has $n=500000$ variables all belonging to $[0,1]$ and $m=5000$ constraints. It has a nonzero density equal to $2.50 \times 10^{-3}$ and is originally solved in $25.65$ seconds. Using $\varepsilon=0.2$ the projected problem has $k=385$ constraints and has been solved in $537.10$ seconds for an approximation ratio equal to  
 $3.79 \times 10^{-1}$. For this instance we notice that it takes much more time to solve the projected problem rather than the original one. One possible explanation is that since the projection matrix is not sparse, the projected problem lose all its sparsity pattern after random projections, which may lead to a greater solving time despite having fewer constraints. Furthermore, the ratio between the number of constraints of the original problem and the projected problem is quite small, compared to the other cases. This might explain why, for this case,  reducing the number of constraints did not reduce the solving time. 

\section{Conclusion and future works}
In this paper we applied random projection to reduce the number of constraints of linear optimization problems over a cone $\mathcal{K}$ which is a product of the non-negative orthant and semidefinite cones, i.e., $\mathcal{K}=\mathbb{R}^m_{+} \times \PSDcone{p_1}\times \cdots \times \PSDcone{p_l}$. We considered a random projection matrix $\mathcal{Q}$ such that $\mathcal{Q}(\mathcal{K})$ is also a product of the non-negative orthant and semidefinite cones, each having smaller dimension. Under some conditions on the original problem, we could obtain some theoretical guarantees on the value of the projected problem.

One possible future work would be to consider sparse projection matrices to further reduce the time to solve the projected problem. One such approach 
would be to consider \textit{Johnson-Lindenstrauss transform}:  

\begin{defi}\label{def:1}
	A random $k \times l$ matrix $T$ is a Johnson-Lindenstrauss transform (JLT) with parameters $(\varepsilon,\delta,h)$ if with probability at least $1-\delta$, for any $h$-element subset $\mathcal{Z}\subset \mathbb{R}^l$, for all $z_1,z_2\in \mathcal{Z}$ we have
	\begin{equation}\label{eq:jll2}
	(1-\varepsilon)\|z_1-z_2\|^2_2 \le \|Tz_1-Tz_2\|^2_2 \le (1+\varepsilon)\|z_1-z_2\|^2_2,
	\end{equation}
\end{defi}

By Lemma \ref{lem:jll0}, Gaussian random matrices are JLT. More generally, in \cite{matousek-jll}, it is proven that random matrices with independent sub-Gaussian entries are also JLT. In \cite[Section 5.1]{d2020random}, 
it is proven that sparse Gaussian matrices, where each entry is non-zero with probability $1-\gamma$, are JLT. However to generalize the approach presented in this paper to JLTs we would need a more general version of Lemma \ref{Zhang}.

Another topic for further research is the case of  linear conic problems where the underlying cone $\mathcal{K}$ is only assumed to be symmetric. This  would, for example, allow us to consider second-order cone constraints or to deal more effectively with the case where there is a large number of small-dimensional SDP constraints. However, such generalizations  would require more sophisticated tools dealing  with random matrices over more general algebras.

\section*{Acknowledgments}
We thank the referee for their  comments, which helped to improve the paper.
This work was partially supported by the Grant-in-Aid for Scientific Research (B) (19H04069) and the Grant-in-Aid for Young Scientists (19K20217) from Japan Society for the Promotion of Science.


\bibliography{dr2}

\end{document}